\documentclass[11pt, reqno, english]{amsart}
\usepackage{comment}
\usepackage{accents}
\usepackage{style}
\usepackage[normalem]{ulem}
\usepackage{wasysym}
\usepackage{charter}
\usepackage{euler}
\usepackage[T1]{fontenc}
\usepackage[bb=stixtwo]{mathalpha}
\usepackage{tikz}     
\usetikzlibrary{arrows,decorations.markings}
\usetikzlibrary{positioning}

\usepackage{caption}
\usepackage{thmtools}
\usepackage{thm-restate}

\begin{document}

\author[Adams]{Henry Adams}
\address[HA]{Department of Mathematics, University of Florida, Gainesville, FL 32611, USA}
\email{henry.adams@ufl.edu}

\author[Carvajal]{Julian Carvajal}
\address[JC]{Department of Mathematics, University of Florida, Gainesville, FL 32611, USA}
\email{j.carvajal@ufl.edu}

\author[Rhodes]{Jake Rhodes}
\address[JR]{Department of Mathematics, University of Florida, Gainesville, FL 32611, USA}
\email{wrhodes1@ufl.edu}

\author[Turillo]{Niccolo Turillo}
\address[NT]{Department of Mathematics, University of Florida, Gainesville, FL 32611, USA}
\email{nturillo@ufl.edu}

\author[Ye]{Jingkai Ye}
\address[JY]{Department of Mathematics, University of Florida, Gainesville, FL 32611, USA}
\email{jye1@ufl.edu}

\author[Ying]{Raymond Ying}
\address[RY]{Department of Mathematics, University of Florida, Gainesville, FL 32611, USA}
\email{raymondying@ufl.edu}

\title{
Vietoris--Rips complexes of ellipses at larger scales
}

\begin{abstract}
For $X$ a metric space and $r>0$, the Vietoris--Rips simplicial complex $\vr{X}{r}$ has $X$ as its vertex set, and a finite subset $\sigma \subseteq X$ as a simplex whenever the diameter of $\sigma$ is less than $r$.
In~\cite{AAR}, the authors studied the homotopy types of Vietoris--Rips complexes of ellipses $E_a=\{(x,y)\in \mathbb{R}^2~|~(x/a)^2+y^2=1\}$ of small eccentricity, meaning $1<a< \sqrt{2}$, at small scales $r < \frac{4\sqrt{3}a^2}{3a^2+1}$.
In this paper, we further investigate the homotopy types that appear at larger scales.
In particular, we identify the scale parameters $r$, as a function of the eccentricity $a$, for which the Vietoris--Rips complex $\vr{E_a}{r}$ is homotopy equivalent to a $3$-sphere, to a wedge sum of $4$-spheres, or to a $5$-sphere.
%Furthermore, we investigate ellipses of eccentricity greater than $\sqrt{2}$ and determine the smallest radius $r$ such that $\vr{E_a}{r}$ is no longer a cyclic graph.
\end{abstract}

\maketitle
  
%\setcounter{tocdepth}{1}
%\tableofcontents

\section{Introduction}
\label{sec:intro}

Let $X$ be a metric space and let $r>0$ be a scale parameter.
The Vietoris--Rips complex of $X$, denoted $\vr{X}{r}$, is an abstract simplicial complex --- a space formed by gluing vertices, edges, triangles, and higher-dimensional simplices together --- with vertex set $X$.
Indeed, we define a finite subset $\sigma\subseteq X$ to be a simplex of $\vr{X}{r}$ whenever $\diam(\sigma)< r$.

In applications, $X$ is usually a finite subset, assumed to have been sampled from a larger space $M$, and $\vr{X}{r}$ is used to obtain topological information about $M$.

Vietoris--Rips complexes are widely used in persistent homology~\cite{EdelsbrunnerHarer,Carlsson2009}, which tracks homological features of $\vr{X}{r}$ as $r$ grows, to see which features persist with the scale parameter.
There are a number of theoretical results that justify using Vietoris--Rips complexes:
\begin{enumerate}
\item
If $M$ is a Riemannian manifold, and $r$ is sufficiently small, then $\vr{M}{r}$ is homotopy equivalent to $M$ itself by Hausmann's theorem~\cite{Hausmann1995}.
\item
If $M$ is a Riemannian manifold, $r$ is sufficiently small, and $X$ is close enough to $M$ in the Gromov--Hausdorff distance (say $X$ is a noisy sample), then Latschev proved that $\vr{X}{r}$ is still homotopy equivalent to $M$~\cite{Latschev2001}.
\item
If $X_1,X_2,\ldots$ are a sequence of samples near a metric space $M$ that converge to $M$ in the Gromov--Hausdorff distance, then the persistent homology of $\vr{X}{r}$ converges to a limiting object, namely the persistent homology of $\vr{M}{r}$~\cite{ChazalDeSilvaOudot2014}.
\end{enumerate}
While (1) and (2) both require that $r$ be sufficiently small, (3) does not.
In practice, the underlying space $M$ is unknown, and thus the values of $r$ which satisfy (1) and (2) are also generally unknown.
While (3) guarantees convergence to some limiting object, little is known about the limiting object---the persistent homology of $\vr{M}{r}$---itself.

To our knowledge, the only compact non-contractible Riemannian manifold for which the homotopy types of its Vietoris--Rips complexes are known for all scales is the circle.
It has been shown that if $S^1$ is the circle, then $\vr{S^1}{r}$ obtains all homotopy types of the form $S^1$, $S^3$, $S^5$, $S^7$,\ldots as $r$ increases, until it is finally contractible~\cite{AA-VRS1}.

Let $E_a=\{(x,y)\in\R^2~|~(x/a)^2+y^2=1\}$ be the ellipse with semi-major axis of length $1 < a < \sqrt{2}$ and with semi-minor axis of length one.
We equip $E_a$ with the Euclidean metric.
With this metric, $E_a$ is a manifold but not a Riemannian manifold.

In 2017, Adamaszek, Adams, and Reddy~\cite{AAR} showed that if $E_a$ is an ellipse of small eccentricity (namely $1<a<\sqrt{2}$), and if $r_1 = \frac{4\sqrt{3}a}{a^2+3}$ and $r_2=\frac{4\sqrt{3}a^2}{3a^2+1}$, then $\vr{E_a}{r}\simeq S^1$ for $0<r\leq r_1$ and $\vr{E_a}{r}\simeq S^2$ for $r_1<r\leq r_2$.

We investigate the homotopy types of Vietoris--Rips complexes of ellipses, $\vr{E_a}{r}$, that appear at larger scales $r > r_2 =\frac{4\sqrt{3}a^2}{3a^2+1}$.
We supply strong empirical evidence supporting a conjecture about inscribed 5-pointed stars inside of the ellipse $E_a$, which, if true, is strong enough to imply the next homotopy types of $\vr{E_a}{r}$ at larger scales.
More specifically, the next possible homotopy types of $\vr{E_a}{r}$ at larger scales include $S^3$, $S^4$, $\bigvee^3 S^4$, and $S^5$.

\begin{figure}
    \centering
    \includegraphics[width=0.45\linewidth]{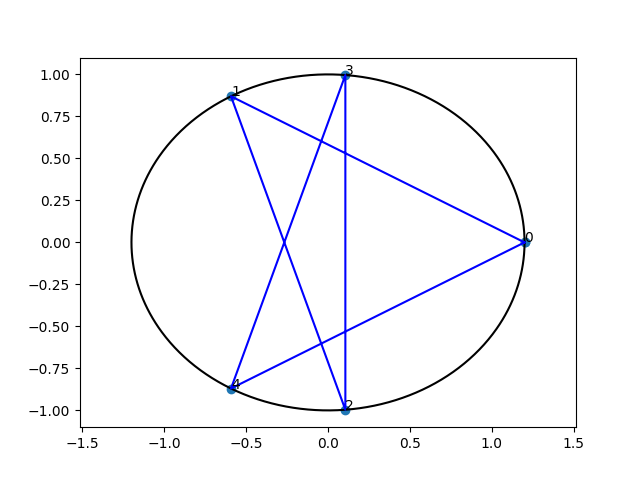}
    \includegraphics[width=0.45\linewidth]{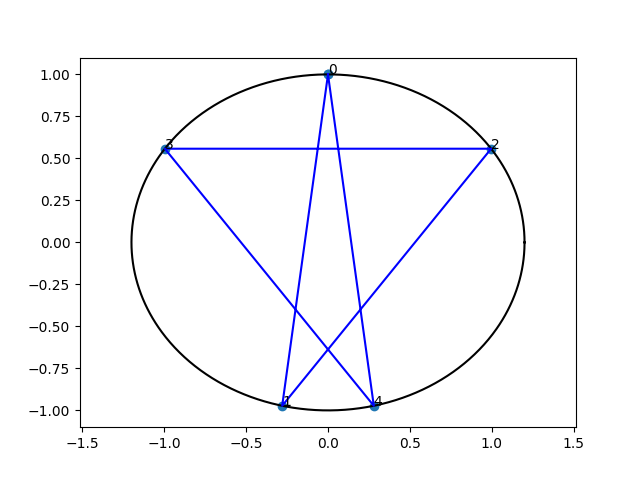}
    \caption{(Left) A minimal inscribed star containing $(a,0)$, and (Right) a maximal inscribed star containing $(0,1)$, both found computationally.}
    \label{fig:minimal-star}
\end{figure}

%\note{Nicco says: do we like having two stars here or should we go back to just one?}
%\note{Henry says: Two stars looks good.}

%Based on experimental evidence shown in Figure~\ref{fig:star-side-lengths}, we conjecture the behavior of the side-length function of inscribed 5-pointed stars in the ellipse.
%The first conjecture is supported by computational evidence.

A main object we study are \emph{$5$-pointed stars wrapping twice around $E_a$}.
Let $S=(p_0,\ldots,p_4)$ be cyclically ordered points on an ellipse $E_a$.
Draw line segments between each $p_i$ and $p_{i+2}$, where all indices are taken modulo $5$.
If all line segments are of equal length, then we call $S$ a \emph{$5$-pointed star}.
Consider the arcs traversed along the ellipse between successive points $p_i$ and $p_{i+2}$ on a star $S$.
If the sum of the total angle traveled on these arcs is $4\pi$, then we say that the star $S$ \emph{wraps twice around $E_a$}; see Figure~\ref{fig:minimal-star}.

Let $1\le a<\sqrt{2}$.
At each point $p\in E_a$, there is a unique $5$-pointed star based at $p$ wrapping twice around.
We define the side length function $s_a\colon E_a\to \R$ by setting $s_a(p)$ equal to the diameter of the star starting at $p$.
(Here, we write $s_a$ with domain $E_a$ with slight abuse of notation.
As we will see later, $s_a$ is formally defined with domain $S^1$, the circle.)

%We can construct stars by choosing a singular point $p\in E_a$ and a length value for the line segments (which we call side lengths) $r\in \R$.
%From there, one can find the (two) unique points that are exactly a distance $r$ from the point $p$.
%We continue in this way, by finding the unique points that are a distance $r$ from our previously constructed points and adding them to our construction.
%If we choose our $r$ value strategically, then our construction is a finite process resulting in an $n$-pointed star containing the point $p$.
%Interestingly, for a given point $p\in E_a$ and a target number of points $n$ that we wish to have on our star, there is exactly one side length value $r_n$ such that the aforementioned construction will result in an $n$-pointed star containing $p$.
%As such, there exists a well-defined function 
%\[ s_a : E_a\times \Z^{\geq 3} \to \R\]
%that maps a point $p$ and a natural number $n\geq 3$ to the corresponding side length value that will ensure the existence of an $n$-pointed star containing the point $p$.

\begin{conjecture}
\label{conj:main}
Let $1< a<\sqrt{2}$.
We conjecture there are real values $1 < a_1^- < a_1 < a_1^+ < a_2^- < a_2 < a_2^+ <\sqrt{2}$ with $a_1\approx 1.3299$ and $a_2\approx 1.4123$ so that the side-length function $s_a\colon E_a\to \R$ satisfies the following.
% In the next version, we should also give approximate values of $a_1^-, a_1^+, a_2^-, a_2^+$.
\begin{itemize}
\item If $1 < a \le a_1^-$, then $s_a$ has ten global minima including $(\pm a,0)$ and ten global maxima including $(0,\pm 1)$.
\item If $a_1^- < a <a_1$, then $s_a$ has ten global minima including $(\pm a,0)$, ten local minima including $(0,\pm 1)$, and twenty global maxima.
\item If $a = a_1$, then $s_a$ has twenty global minima including $(\pm a,0)$ and $(0,\pm 1)$ and twenty global maxima.
\item If $a_1 < a < a_1^+$, then $s_a$ has ten global minima including $(0,\pm 1)$, ten local minima including $(\pm a,0)$, and twenty global maxima.
\item If $a_1^+ \le a \le a_2^-$, then $s_a$ has ten global minima including $(0,\pm 1)$ and ten global maxima including $(\pm a,0)$.
\item If $a_2^- < a < a_2$, then $s_a$ has ten global minima including $(0, \pm 1)$, ten local minima including $(\pm a, 0)$, and twenty global maxima.
\item if $a = a_2$ then $s_a$ has twenty global minima including $(0, \pm 1)$ and $(\pm a, 0)$, and twenty global maxima.
\item if $a_2 < a < a_2^+$ then $s_a$ has ten global minima including $(\pm a, 0)$, ten local minima including $(0, \pm 1)$, and twenty global maxima.
\item if $a_2^+ \le a < \sqrt{2}$, then $s_a$ has ten global minima including $(\pm a, 0)$ and ten global maxima including $(0, \pm 1)$.
\end{itemize}
In all cases, $s_a$ is monotonic between adjacent extrema mentioned above.
\end{conjecture}

See Figure~\ref{fig:star-side-lengths} for experimental evidence for this conjecture.

Intuitively, we expect $(\pm a, 0)$ and $(0, \pm 1)$ to be extrema of $s_a$ by symmetry.
A more surprising consequence of our conjecture is that for some specific eccentricities $a_1 \approx 1.3299, a_2 \approx 1.4123$, these are not the only extrema.
In other words, for these specific eccentricities, there exist extrema corresponding to stars that do not contain points on either axis.
As a partial result in support of Conjecture~\ref{conj:main}, we make and prove the following theorem, which captures some of the surprising behavior.

\begin{theorem}
\label{thm:extrema}
There exist two distinct $a_1, a_2 \in (1, \sqrt{2})$ (with $a_1\approx 1.3299$ and $a_2 \approx 1.4123$) such that $s_a$ has at least twenty local minima and twenty local maxima.
\end{theorem}

Let $1<a<\sqrt{2}$, let $r_1(a)=\frac{4\sqrt{3}a}{a^2+3}$, and let $r_2(a)=\frac{4\sqrt{3}a^2}{3a^2+1}$.
We know that these functions $r_1,r_2\colon \R\to \R$ give the global minimum and maximum side-length of inscribed equilateral triangles.
From~\cite{AAR}, we know that
\[\vr{E_a}{r}\simeq\begin{cases}
S^1 & \text{if }0<r\le r_1(a) \\
S^2 & \text{if }r_1(a)<r\le r_2(a).
\end{cases}\]
Assuming Conjecture~\ref{conj:main} is true, we prove the homotopy types of Vietoris--Rips complexes of ellipses at larger scale parameters.

\begin{theorem}
\label{thm:main}
Let $1 < a < \sqrt{2}$.
If Conjecture~\ref{conj:main} is true, then there are functions $r_3,r_4,r_5\colon \R\to \R$ (global minimum side-length of $5$-pointed stars, global maximum side-length of $5$-pointed stars, and global minimum side-length of $7$-pointed stars) such that
\begin{itemize}
\item If $1 < a \le a_1^-$ or $a_1^+ \le a \le a_2^-$ or $a_2^+ \le a < \sqrt{2}$, then 
\[\vr{E_a}{r}\simeq\begin{cases}
S^3 & \text{if }r_2(a)<r\le r_3(a) \\
S^4 & \text{if }r_3(a)<r\le r_4(a) \\
S^5 & \text{if }r_4(a)<r\le r_5(a).
\end{cases}\]
\item If $a = a_1$ or $a = a_2$, then 
\[\vr{E_a}{r}\simeq\begin{cases}
S^3 & \text{if }r_2(a)<r\le r_3(a) \\
\bigvee^3 S^4 & \text{if }r_3(a)<r\le r_4(a) \\
S^5 & \text{if }r_4(a)<r\le r_5(a).
\end{cases}\]
\item If $a_1^- < a <a_1$ or $a_1 < a < a_1^+$ or $a_2^- < a < a_2$ or $a_2 < a < a_2^+$, then there is an additional function $r_{7/2}\colon \R\to \R$ (local minimum side-length of $5$-pointed stars) so that
\[\vr{E_a}{r}\simeq\begin{cases}
S^3 & \text{if }r_2(a)<r\le r_3(a) \\
S^4 & \text{if }r_3(a)<r\le r_{7/2}(a) \\
\bigvee^3 S^4 & \text{if }r_{7/2}(a)<r\le r_4(a) \\
S^5 & \text{if }r_4(a)<r\le r_5(a).
\end{cases}\]
\end{itemize}
Furthermore, the inclusion $\vr{E_a}{r}\hookrightarrow\vr{E}{\tilde{r}}$
\begin{itemize}
\item is a homotopy equivalence for any $r_2(a)<r\le \tilde{r}\le r_3(a)$ or any $r_4(a)<r\le \tilde{r}\le r_5(a)$;
\item is a homotopy equivalence for any $r_3(a)<r\le \tilde{r}\le r_4(a)$ if $1 < a \le a_1^-$, $a = a_1$, $a_1^+ \le a \le a_2^-$, $a = a_2$, or $a_2^+ \le a < \sqrt{2}$;
\item is a homotopy equivalence for any $r_3(a)<r\le \tilde{r}\le r_{7/2}(a)$ or any $r_{7/2}(a)<r\le \tilde{r}\le r_4(a)$ if $a_1^- < a <a_1$, $a_1 < a < a_1^+$, $a_2^- < a < a_2$, or $a_2 < a < a_2^+$;
\item induces a rank $1$ map on 4-dimensional homology $H_4(-;\mathbb{F})$ for any field $\mathbb{F}$ for any $r_3(a)<r\le r_{7/2}(a)<\tilde{r}\le r_4(a)$ if $a_1^- < a <a_1$, $a_1 < a < a_1^+$, $a_2^- < a < a_2$, or $a_2 < a < a_2^+$.
\end{itemize}
\end{theorem}

\begin{figure}[b]
\begin{center}
\begin{tikzpicture}[x=1cm,y=2cm]
  % Axis
  \draw[->] (-.5,0) -- (10,0) node[right] {$r$};
  % Axis labels
  \node[below] at (0,0) {$r_2(a)$};
  \node[below] at (3,0) {$r_3(a)$};
  \node[below] at (4.5,0) {$r_{7/2}(a)$};
  \node[below] at (6,0) {$r_4(a)$};
  \node[below] at (9,0) {$r_5(a)$};
  % Tick Marks
  \draw[line width=1pt] (0, -0.05) -- (0, 0.05);
  \draw[line width=1pt] (3, -0.05) -- (3, 0.05);
  \draw[line width=1pt] (4.5, -0.05) -- (4.5, 0.05);
  \draw[line width=1pt] (6, -0.05) -- (6, 0.05);
  \draw[line width=1pt] (9, -0.05) -- (9, 0.05);
  % H0 bars
  \draw[line width=1pt]  (0,0.2) -- (9,0.2);
  % H3 bars
  \draw[line width = 1pt] (0,0.4) -- (3, 0.4);
  \node[above] at (1.5, 0.45) {$S^3$};
  % H4 bars
  \draw[line width = 1pt] (3,0.6) -- (6,0.6);
  \node[above] at (3.75, 0.65) {$S^4$};
  \draw[line width = 1pt] (4.5, 0.8) -- (6, 0.8);
  \draw[line width = 1pt] (4.5, 1.0) -- (6, 1.0);
  \node[above] at (5.25, 1.05) {$\bigvee^3 S^4$};
  % H5 bars
  \draw[line width = 1pt] (6, 1.2) -- (9, 1.2);
  \node[above] at (7.5, 1.25) {$S^5$};
\end{tikzpicture}
\caption{The homotopy types of $\vr{E_a}{r}$ as $r$ varies, for $a_1^- < a <a_1$, $a_1 < a < a_1^+$, $a_2^- < a < a_2$, or $a_2 < a < a_2^+$.}
\label{fig:barcodes}
\end{center}
\end{figure}
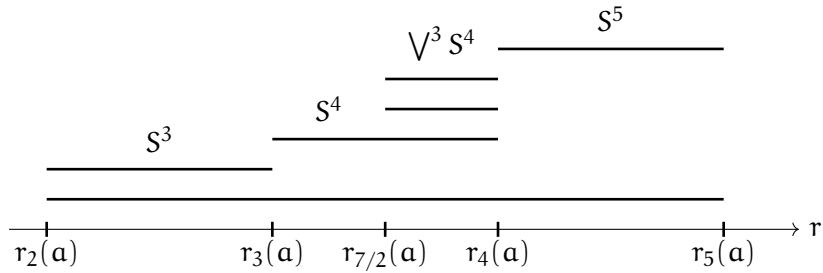

See Figure~\ref{fig:barcodes}.

We begin in Section~\ref{sec:relatedwork} with related work.
In Section~\ref{sec:preliminaries} we discuss the preliminaries related to this paper.
Then, we define the side-length function in Section~\ref{sec:sidelength} and discuss the evidence related to Conjecture~\ref{conj:main}.
Furthermore, we prove Theorem~\ref{thm:extrema} in Section~\ref{sec:extrema} and Theorem~\ref{thm:main} in Section~\ref{sec:maintheorem}.
In Section~\ref{sec:JointContF} we study the function that moves a point to the next clockwise point a fixed distance away, and prove the joint continuity of this function.
We conclude in Section~\ref{sec:conclusion} with a list of open questions.

\section{Related work}
\label{sec:relatedwork}

The Vietoris--Rips complex $\vr{X}{r}$ is often studied in the context of $X$ sampled from a manifold $M$.
A seminal result by Hausmann~\cite{Hausmann1995} showed that if $M$ is a compact Riemannian manifold, and $\epsilon > 0$ is sufficiently small compared to the curvature of $M$, then the Vietoris--Rips complex $\vr{M}{\epsilon}$ is homotopy equivalent to $M$.
Latschev~\cite{Latschev2001} extended Hausmann's work with the following result.
If $M$ is a closed Riemannian manifold, then there exists $\epsilon_0 > 0$ such that for all $0 < \epsilon \leq \epsilon_0$, there is $\delta > 0$ with the property that if the Gromov--Haursdorff distance between a metric space $X$ and $M$ is less than $\delta$, then $\vr{X}{\epsilon}$ is homotopy equivalent to $M$.
Following Hausmann and Latschev, Chazal et al.~\cite{ChazalDeSilvaOudot2014} show that for any two totally bounded metric spaces, their Vietoris--Rips filtrations are $\epsilon$-interleaved for any $\epsilon$ that is at least twice the Gromov--Hausdorff distance between them, and hence as a consequence the bottleneck distance between corresponding persistence diagrams is at most $\epsilon$.
So, if $X_1, X_2, X_3, \ldots$ is a sequence of datasets converging to some metric space $M$ in the Gromov--Hausdorff distance, then the persistent homology of $\vr{X_n}{-}$ converges to that of $\vr{M}{-}$ as $n\to \infty$.

The homotopy type of $\vr{X}{r}$ is not well understood for general $X$ and $r$, but when $X$ is the circle, the homotopy type is known for all values of $r$.
Indeed, as $r>0$ increases from $0$, the homotopy type of $\vr{S^1}{r}$ ranges from $S^1$ to $S^3$ to $S^5$, and so on, achieving the homotopy type of every odd dimensional sphere until eventually becoming contractible~\cite{AA-VRS1}.
To prove this result, they introduce the notion of a winding fraction for directed graphs.
For a survey of what is known about the Vietoris--Rips complexes of higher-dimensional spheres, see~\cite{ABV}.

We now turn our attention to the case when $X = E_a$ is an ellipse, which is the focus of this paper.
Some results on $\vr{E_a}{r}$ are known for small values of $r$ and $a$.
In~\cite{AAR}, Adamaszek, Adams, and Reddy investigate homotopy types of the Vietoris--Rips complexes of ellipses and find that the $1$-skeleton of $\vr{E_a}{r}$ can be modeled as a cyclic graph whenever $1 < a \leq \sqrt{2}$.
We say that an ellipse $E_a$ has \emph{small eccentricity} when $1 < a \leq \sqrt{2}$.
Theorem~\cite[Theorem~7.3]{AAR} states that if $E_a$ is an ellipse with small eccentricity, then
\[\vr{E_a}{r} \simeq
\begin{cases}
S^1 ~\text{for} ~0<r\leq r_1(a) \\
S^2 ~\text{for} ~r_1(a)<r\leq r_2(a),
\end{cases}\]
where $r_1(a) = \frac{4\sqrt{3}a}{a^2+3}$, and $r_2(a) = \frac{4\sqrt{3}a^2}{3a^2+1}$.
For a more in-depth treatment on the tools of cyclic graphs, cyclic dynamical systems, and winding fractions which we use in this paper, see~\cite{AAR}.

\section{Preliminaries}
\label{sec:preliminaries}

\subsection*{Metric spaces}

Let $(X,d)$ be a metric space with metric $d \colon X \times X \to \mathbb{R}_{\geq 0}$.

The \emph{diameter} of a finite subset $\sigma \subseteq X$ is $\text{diam}(\sigma) \coloneq \displaystyle \max_{x,y \in \sigma} d(x,y)$.

\subsection*{Simplicial complexes}

An \emph{abstract simplicial complex} is a pair $K=(V,E)$, where $V$ is a set of vertices, and $E$ is a collection of finite subsets of $V$ called simplices, such that if $\sigma\in E$ and $\tau\subseteq \sigma$, then $\tau\in E$.

Let $K=(V,E)$ be an abstract simplicial complex.
Then for any $p$-simplex $\sigma\in E$, we can define its \emph{geometric realization} $|\sigma|$ to be the convex hull of the standard unit vectors in $\R^{p+1}$, 
\[ |\sigma| = \left\{ \sum_{i=1}^{p+1} \alpha_i e_i \colon \alpha_i > 0, \sum_{i=1}^{p+1} \alpha_i = 1 \right\} \]
To define the geometric realization of the whole entire simplicial complex, we define $|K|$ to be the collection of all functions $\alpha:V\to [0,1]$, such that $\{v\in V~|~\alpha(v)\neq 0\}\in E$, and $\sum_{v\in V}\alpha(v)=1$.
We endow this set with the quotient topology inherited from $\coprod_{\sigma\in E}|\sigma|\to |K|$.
The definition of this topology is complicated so that it works even when the vertex set $V$ is infinite, but the idea is essentially to place barycentric coordinates on simplices.

Throughout the rest of this paper, we often identify abstract simplicial complexes with their geometric realizations.

\subsection*{Vietoris--Rips simplicial complexes}

Let $X$ be a metric space.
For scale parameter $r > 0$, the \emph{Vietoris--Rips simplicial complex $\vr{X}{r}$} consists of all finite subsets $\sigma \subseteq X$ with diam$(\sigma) < r$.

\subsection*{The circle}

The circle, denoted $S^1$, is thought of as the set of all unit vectors in $\R^2$, equipped with the subspace topology inherited from $\R^2$.
We occasionally parameterize $S^1$ using polar coordinates, referring to points $x\in S^1$ as $x=(\cos\theta, \sin\theta)$ for $\theta\in [0,2\pi)$.

Additionally, it is useful to define open and closed clockwise arcs on the circle.
We use %the order relation 
$\preceq$ to denote the clockwise order that vertices appear on the ellipse.
In particular, we say $p_1 \preceq p_2 \preceq \ldots \preceq p_s \preceq p_1$ if $p_2$ appears clockwise from $p_1$ or is equal to $p_1$ in $S^1$, $p_3$ appears clockwise from $p_2$ or is equal to $p_2$, and so forth.
We denote the closed \emph{clockwise} arc on $S^1$ by $[p,q]_{S^1}=\{x\in S^1 ~:~p \preceq x \preceq q \preceq p\}$.
Open and half-open arcs are defined similarly and are denoted $(p,q)_{S^1}$, $[p,q)_{S^1}$, and $(p,q]_{S^1}$.
The directed clockwise distance on $S^1$ is defined by $\vec{d}:S^1\times S^1\to [0,2\pi)$, where $\vec{d}(p,q)$, where is the length of the closed clockwise arc from $p$ to $q$.
Note that $\vec{d}(p,q) \neq \vec{d}(q,p)$ for arbitrary $p\neq q$ and $p\neq -q$.

\subsection*{Ellipses}

For a real number $a > 1$, we define an \emph{ellipse} $E_a$ as
\[ E_a \coloneqq \{(x,y)\in \mathbb{R}^2 \colon (x/a)^2 + y^2 = 1\}\]
equipped with the restriction of the Euclidean metric on $\R^2$.
When $1 < a < \sqrt{2}$, we say the ellipse $E_a$ is of \emph{small eccentricity}.
%When $a$ is clear from context, or irrelevant, we simply denote an ellipse by $E$.
Note that we can define a clockwise order on $E_a$ by saying $p\preceq q \preceq p$ for $p,q\in E_a$ if $\varphi^{-1}_a(p) \preceq \varphi^{-1}_a(q) \preceq \varphi^{-1}_{a}(p)$ in $S^1$.
We frequently use the family of homeomorphisms $\varphi_a:S^1\to E_a$ that map $\varphi_a:(\cos\theta, \sin\theta)\mapsto (a\cos\theta,\sin\theta)$.
Observe that $\varphi_a$ preserves the clockwise order on $S^1$ and $E_a$.
That is, if $x,y\in S^1$ satisfy $x\preceq y \preceq x$, then $\varphi_a(x) \preceq \varphi_a(y) \preceq \varphi_a(x)$.
Frequently, we define maps on the ellipse, but their domain and/or codomain is $S^1$.
Examples of this include the side length function $s_a : S^1 \rightarrow \R$ and the step function $F : S^1 \times (1, 2)_{\R} \times (1, \sqrt{2})_{\R}$ which we examine closesly in Section~\ref{sec:JointContF}.
In this case, assume the maps are composed with $\varphi_a$ or $\varphi_a^{-1}$ as needed.

We define the open, closed, and half-open clockwise arcs on $E_a$ analogously to how they are defined for $S^1$ and denote them $(p,q),~ [p,q],~(p,q],$ and $[p,q)$ for $p \preceq q \preceq p\in E_a$.
%\note{Henry says: I don't get the following two sentences.
%But we can just delete them, right?
%We just use $\vec{d}$; we never use $\vec{d}_a$.}
%Next, the directed distance on $E_a$, denoted $\vec{d}_a\colon S^1\times S^1\to \R$ is defined for $x\preceq y \preceq x\in S^1$, 
%\[
%\vec{d}_a(x,y) = L([\varphi_a(x),\varphi_a(y)]).
%\]
%That is, the $\vec{d}_a(x,y)$ is the clockwise arc length on $E_a$ of the arc $[\varphi_a(x),\varphi_a(y)]$.

\subsection*{Directed graphs}

A \emph{homomorphism of directed graphs} $h : G \to \tilde{G}$ is a vertex map such that for every edge $v \to w$ in $G$, either $h(v) = h(w)$ or there is an edge $h(v) \to h(w)$ in $\tilde{G}$.
The \emph{out-neighborhood} of a vertex $v$ is defined as $N^+(G,v) = \{w~|~ v \to w\}$, and its closed version is $N^+[G,v] = \{w~|~ v \to w\}\cup\{v\}$.

\subsection*{Cyclic graphs}

A large part of the remaining preliminaries comes from~\cite{AAR}, where proofs may be found.
A directed graph with vertex set $V \subseteq S^1$ is \emph{cyclic} if whenever there is a directed edge $v \to u$, then there are directed edges $v \to w$ and $w \to u$ for all $v \prec w \prec u \prec v$.

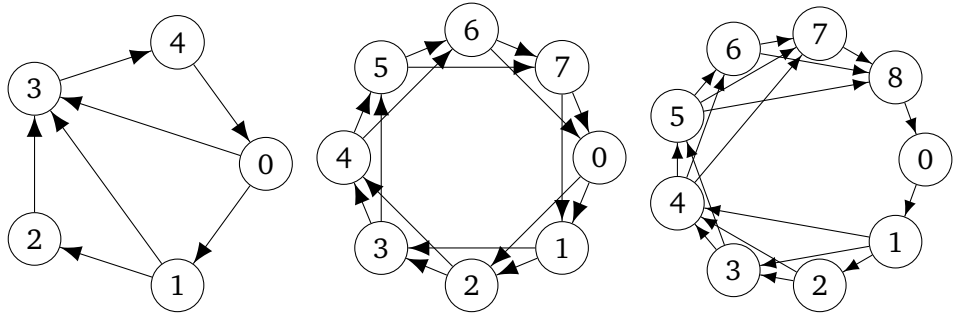
\begin{figure*}[h]
    \begin{center}
 \begin{tikzpicture}[
     decoration={
       markings,
       mark=at position 1 with {\arrow[scale=2,black]{latex}};
     },every node/.style={draw,circle}
   ]
\def \n {4}
\def \radius {3cm}
\def \margin {8} % margin in angles, depends on the radius
\foreach \s in {0,...,4}
{\draw (-72*\s: 1.7cm) node(\s){\s};}
    \draw [postaction=decorate] (0) -- (1);
    \draw [postaction=decorate] (0) -- (3);
    \draw [postaction=decorate] (1) -- (2);
    \draw [postaction=decorate] (1) -- (3);
    \draw [postaction=decorate] (2) -- (3);
    \draw [postaction=decorate] (3) -- (4);
    \draw [postaction=decorate] (4) -- (0);
\end{tikzpicture}
\hspace{1mm}
 \begin{tikzpicture}[
     decoration={
       markings,
       mark=at position 1 with {\arrow[scale=2,black]{latex}};
     },every node/.style={draw,circle}
   ]
\def \n {5}
\def \radius {3cm}
\def \margin {8} % margin in angles, depends on the radius
\foreach \s in {0,...,7}
{\draw (-45*\s: 1.7cm) node(\s){\s};}
    \draw [postaction=decorate] (0) -- (1);
    \draw [postaction=decorate] (0) -- (2);
    \draw [postaction=decorate] (1) -- (2);
    \draw [postaction=decorate] (1) -- (3);
    \draw [postaction=decorate] (2) -- (3);
    \draw [postaction=decorate] (2) -- (4);
    \draw [postaction=decorate] (3) -- (4);
    \draw [postaction=decorate] (3) -- (5);
    \draw [postaction=decorate] (4) -- (5);
    \draw [postaction=decorate] (4) -- (6);
    \draw [postaction=decorate] (5) -- (6);
    \draw [postaction=decorate] (5) -- (7);
    \draw [postaction=decorate] (6) -- (7);
    \draw [postaction=decorate] (6) -- (0);
    \draw [postaction=decorate] (7) -- (0);
    \draw [postaction=decorate] (7) -- (1);
\end{tikzpicture}
\hspace{1mm}
\begin{tikzpicture}[
     decoration={
       markings,
       mark=at position 1 with {\arrow[scale=1.5,black]{latex}};
     },every node/.style={draw,circle}
   ]
\def \n {5}
\def \radius {3cm}
\def \margin {8} % margin in angles, depends on the radius
\foreach \s in {0,...,8}
{\draw (-40*\s: 1.7cm) node(\s){\s};}
    \draw [postaction=decorate] (0) -- (1);
    \draw [postaction=decorate] (1) -- (2);
    \draw [postaction=decorate] (1) -- (3);
    \draw [postaction=decorate] (1) -- (4);
    \draw [postaction=decorate] (2) -- (3);
    \draw [postaction=decorate] (2) -- (4);
    \draw [postaction=decorate] (3) -- (4);
    \draw [postaction=decorate] (3) -- (5);
    \draw [postaction=decorate] (4) -- (5);
    \draw [postaction=decorate] (4) -- (6);
    \draw [postaction=decorate] (4) -- (7);
    \draw [postaction=decorate] (5) -- (6);
    \draw [postaction=decorate] (5) -- (7);
    \draw [postaction=decorate] (5) -- (8);
    \draw [postaction=decorate] (6) -- (7);
    \draw [postaction=decorate] (6) -- (8);
    \draw [postaction=decorate] (7) -- (8);
    \draw [postaction=decorate] (8) -- (0);
    %more edges here
\end{tikzpicture}
\end{center}
\caption{(Left) Not a cyclic graph because there is an edge $0 \to 3$ but not $0 \to 2$.
(Middle) The cyclic graph $C_8^2$.
(Right) This cyclic graph has winding fraction $\frac{1}{5}$.}
\label{fig:cyclic}
\end{figure*}

For integers $n$ and $k$ with $0 \leq k \leq \frac{n}2$, the cyclic graph $C_n^k$ has vertex set {0,$\dots$, $n-1$} and edges $i \to (i+s) \mod n$ for all $i=0, \dots, n-1$ and $s = 0, \dots, k$

Let $G$ and $\tilde{G}$ be cyclic graphs.
A homomorphism $h \colon G \to \tilde{G}$ is \emph{cyclic} if
\begin{itemize}
\item whenever $v \prec w \prec u \prec v$ in $G$, then $h(v) \preceq h(w) \preceq h(u) \preceq h(v)$ in $\tilde{G}$, and
\item $h$ is not constant whenever $G$ has a directed cycle of edges.
\end{itemize}

We will use the following numerical invariant of cyclic graphs extensively throughout this paper.
The \emph{winding fraction} of a finite cyclic graph $G$ is \[\text{wf}(G) = \sup\left\{\tfrac{k}{n}~:~\text{there exists a cyclic homomorphism } C_n^k \to G\right\}.\]
The \emph{winding fraction} of a (possibly infinite) cyclic graph $G$ is 
\[\text{wf}(G) = \sup_{W\subseteq V \text{ finite}} \text{wf}(G[W]).\]
We say the \emph{supremum is attained} if there exists some finite $W$ with wf($G$) = wf($G[W]$).
%For finite cyclic graphs, we label the vertices in cyclic order $v_0 \prec v_1 \prec \cdots \prec v_{n-1}$.

\subsection*{Finite cyclic dynamical systems}

Let $G$ be a finite cyclic graph with vertex set $V \subseteq S^1$.
The associated \emph{finite cyclic dynamical system} is generated by the dynamics $f:V\to V$, where $f(v)$ is defined to be the clockwise most vertex of $N^+[G,v]$.

A vertex $v$ is \emph{periodic} if $f^i(v) = v$ for some $i$.
If $v$ is periodic then we refer to $\{f^i(v) | i\geq0\}$ as a \emph{periodic orbit}.
The \emph{length} $\ell$ of a periodic orbit is the smallest $i\geq1$ such that $f^i(v) = v$.
The \emph{winding number} of a periodic orbit is
\[ \omega \coloneqq \sum_{i=0}^{\ell-1}\vec{d}(f^i(v), f^{i+1}(v)) \]
Every periodic orbit in a finite cyclic graph has the same length $\ell$ and winding number $\omega$.
For a finite cyclic graph $G$, we have that $\wf(G) = \frac\omega{\ell}$.

\subsection*{Infinite cyclic graphs}
%Not all infinite cyclic graphs have a corresponding dynamical system.
%Below we outline dynamical systems for infinite cyclic graphs.
%\note{Nicco says: I'm kind of confused by this intro.
%Are we creating a dynamical system when there is none? Or are we saying we're only dealing with cyclic graphs who do have a corresponding dyn system?}
%\note{Henry will put a different intro (not trying to make this subtle point) here.}

Let $G$ be a (possibly infinite) cyclic graph with vertex set $V \subseteq S^1$, and let $m \geq 1$.
We define the map $\gamma_m : V \to \mathbb{R}$ by 

\[\gamma_m(v_0) = \displaystyle \sup\left\{\sum_{i=0}^{m-1} \vec{d}(v_i, v_{i+1})~|~\text{there exist } v_0\to v_1 \to \cdots \to v_m \right\}.\]

We say a cyclic graph $G$ is \emph{closed} if its vertex set $V$ is closed in $S^1$.
We say $G$ is \emph{continuous} if $\gamma_m \colon V \to \mathbb{R}$ is continuous for all $m \geq 1$.
We say that the \emph{supremum is attained} if there exists a directed path $v_0\to v_1 \to \ldots \to v_m$ in $G$ such that
$$
\sum_{i=0}^{m-1}\vec{d}(v_i,v_{i+1})=\gamma_{m}(v_0).
$$
Note since $\gamma_m$ is defined in terms of a supremum, and $G$ is possibly infinite, there may not exist a vertex exactly $\gamma_m$ distance clockwise from $v \in V$ in general.
Hence, we define the map $f_m : V \to S^1$ to the circle rather than to $V$ by
\[ f_m(v) = (v + \gamma_m(v)) \mod 2\pi \]
where the sum is modulo the circumference of the circle.

If $N^+[G,v]$ is closed for all $v \in V$, then we can define a dynamical system for $G$ like for finite cyclic graphs, and $f_m(v) = f^m(v)$.

Now we briefly restrict our attention to the case where $\wf(G) = \frac{p}{q}$ is rational.
We say $v \in V$ is \emph{periodic} if there is some $i \geq 1$ such that $\gamma_i(v) \in \mathbb{N}$ and the supremum defining $\gamma_i(v)$ is achieved.
As in the finite case, the length $l$ of this periodic orbit is the smallest such $i$ and its winding number is $\gamma_l(v)$.
Let $G$ be cyclic with wf$(G) = \frac{p}{q}$.
If $v \in V$ is periodic, then $v$ has orbit length $q$ and winding number $\gamma_q(v) = p$.
Let $V_P \subseteq V$ be the set of all periodic points.
We say a non-periodic vertex $v \in V \backslash V_P$ is \emph{fast} if $\gamma_q(v) > p$, and slow if $\gamma_q(v) \leq p$.
When $G$ is infinite, it is possible for a non-periodic slow vertex to satisfy $\gamma_q(v) = p$.

A fast vertex $v \in V$ \emph{achieves periodicity} if there is some $m$ such that $f_m(v)$ is in $V$, $f_m(v)$ is periodic, and the supremum defining $\gamma_m(v)$ is obtained.
Otherwise we say $v$ is \emph{permanently fast}.

We will often associate a Vietoris--Rips complex to its $1$-skeleton, which is a graph.
In this case of ellipses of small eccentricity, orienting this graph clockwise makes it a cyclic graph.
We then use the machinery of cyclic graphs to analyze the homotopy types of Vietoris--Rips complexes of ellipses.
If $\wf(\vr{E_a}{r}) = \frac{l}{2l+1}$ and the supremum is attained for some $l\in \N$, then we call that Vietoris--Rips complex \emph{singular}.

\section{The side-length function and evidence for Conjecture~\ref{conj:main}}
\label{sec:sidelength}

As discussed in the introduction, we can define a side length function $s_a$ by sending a point $p \in E_a$ to the diameter of the star based at $p$.
The side length function will play a major role in our analysis, and so we take this time to define the function rigorously.
In the following, we note that there is a natural homeomorphism between ellipses and $S^1$.
Hence, fixing an eccentricity $a$, when we speak about a point $p\in S^1$ we are actually interpreting $p$ as a point on $E_a$.
We define a dynamical system $F_a:S^1\times (0,2)\to S^1$ which maps a point $p\in S^1$ and a given radius $r$ to the clockwise first point $q\in S^1$ such that the Euclidean distance on $E_a$ from $\varphi_a(p)$ to $\varphi_a(q)$ is equal to $r$.
For $i \geq 2$, we define $F^i_a$ recursively to be the $i$-fold composition of $F_a$.
That is, $F^i_a(p,r) = F_a (F^{i-1}_a(p,r),r)$ where $F^1_a=F_a$.
We refer to Section~\ref{sec:JointContF} for a full treatment on this dynamical system.

With this definition, we define the side length function $s_a\colon S^1\to \R$ to be
\[
s_a(p)=\inf\left\{r\in \R: \sum_{i=1}^5 \vec{d}(F^i_a(p,r),F^{i+1}_a(p,r)) \ge 4 \pi \right\}.
\]
% Perhaps in a future version of this, we should prove the connection between this definition and the geometric interpretation as a solution to a system of polynomial equations
We note that these definition can naturally be extended to allow for varying eccentricities.
We can thus define $F\colon S^1\times (0,2)\times (1,\sqrt{2})\to S^1$ by $F(p,r,a) = F_a(p,r)$.
Likewise, we have the side length function $s\colon S^1\times (1,\sqrt{2}) \to \R$ defined by $s(p,a) = s_a(p)$.
In Section~\ref{sec:JointContF}, we define this side length function in more generality, but since our current focus is on $5$-pointed stars wrapping twice around an ellipse, this current definition will suffice.

In Figure~\ref{fig:star-side-lengths}, we calculate $s_a(p)$ for $p$ in the first quadrant.
The global/local minima of these graphs correspond to the smallest diameters at which a point on the ellipse will be the base of a $5$-pointed star.
As such, the side length function helps us in determining the homotopy types of the Vietoris--Rips complex as the radius scale changes.

For Conjecture~\ref{conj:main} it sufficed to know the values of $r$ for which the step function with distance $r$ traveled at least twice around the ellipse in five steps.
By symmetry of the ellipse it sufficed to examine only the points in the first quadrant, which we represented by the angle $\theta$ of the point on the ellipse.

We do not find a closed formula for $s_a$, but we are able to closely approximate it using a simulation.
Our methodology for approximating $s_a(p)$ given $p \in S^1$ is as follows.
We know a priori that $0 \leq s_a(p) \leq 2$ since the minor axis of $E_a$ has length 2.
Then, using a black-box function which says how many loops of the ellipse are traveled by taking five steps from $p$ of length $r$, we perform a binary search on $0 \leq r \leq 2$ to find the minimum $r$ such that the number of loops traveled is two.
The black-box function approximates points of counterclockwise distance using a quartic solver.

Once we approximate the $s_a(p)$ for fixed $a$ and $p$, we find the global behavior by sampling evenly spaced points from the ellipse.
The resulting plots are in Figure~\ref{fig:star-side-lengths}.
% We used the code (\footnote{https://www.johndcook.com/blog/2022/11/02/ellipse-rng/}) to generate evenly spaced points based on arclength.
From these plots, we determine the cardinality of $s_a^{-1}(r)$ at each $r$.
The cardinality of $s_a^{-1}(r)$ is important because if the winding fraction $\wf(X;r)$ is $\frac{l}{2l+1}$ for some integer $l$, then the homotopy type of $\vr{E_a}{r}$ is determined by the cardinality of $s_a^{-1}(r)$.
Previous work showed that the side length function of an inscribed triangle, not an inscribed star, had consistent behavior regardless of $a \in (1, \sqrt{2})$.
Specifically, the largest triangles are always based at the north and south poles, and the smallest triangles are always based at the east and west poles, and there are no other extrema present~\cite[Theorem~6.1]{AAR}.
In the present case of the $5$-pointed star, we surprisingly find that there are different behaviors of $s_a$ depending on the value of the eccentricity $a$.
We observe 5 different situations; see Figure~\ref{fig:star-side-lengths}.

\begin{itemize}
\item Situation 1: $s_a$ has ten global minima corresponding to the stars on the east and west poles, ten global maxima corresponding to the stars on the north and south poles, and no other extrema.
\item Situation 2: $s_a$ has ten global minima corresponding to the stars on the east and west poles, ten local minima corresponding to the stars on the north and south poles, and twenty global maxima not corresponding to stars on the poles.
\item Situation 3: $s_a$ has twenty global minima corresponding to the stars on the poles and twenty global maxima not corresponding to stars on the poles.
\item Situation 4: $s_a$ has ten local minima corresponding to the stars on the east and west poles, ten global minima corresponding to the stars on the north and south poles, and twenty global maxima not corresponding to stars on the poles.
\item Situation 5: $s_a$ has ten global maxima corresponding to the stars on the east and west poles, ten global minima corresponding to the stars on the north and south poles, and no other extrema.
\end{itemize}

\begin{figure}
\includegraphics[width=3.2in]{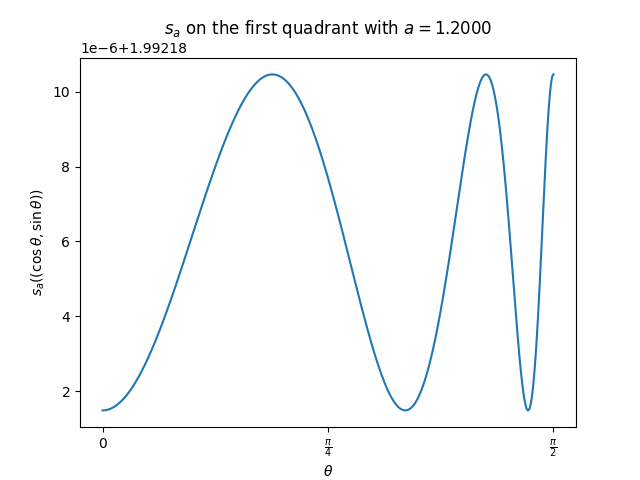}
\includegraphics[width=3.2in]{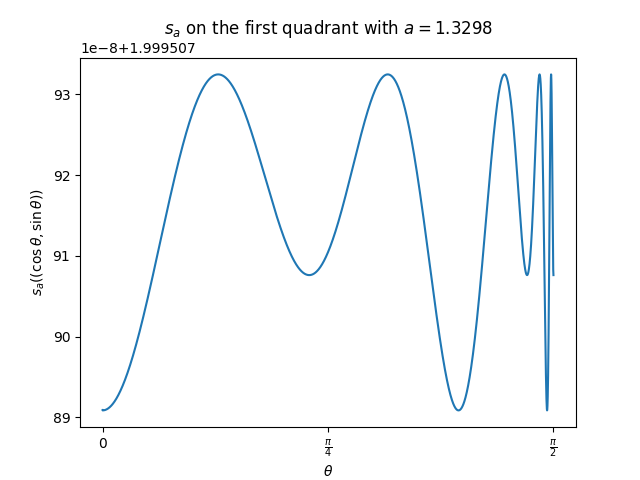}
\includegraphics[width=3.2in]{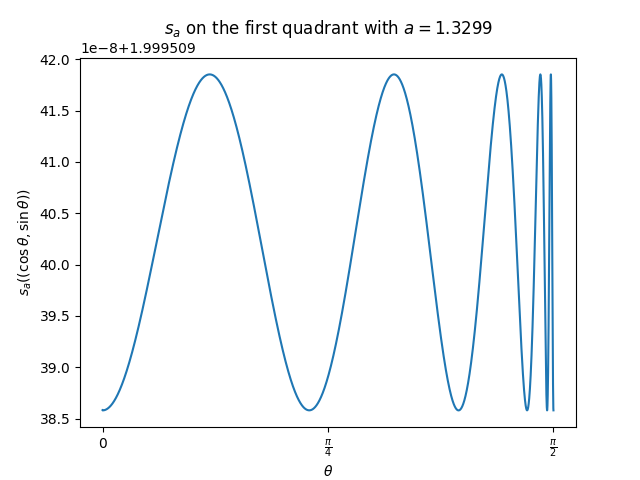}
\includegraphics[width=3.2in]{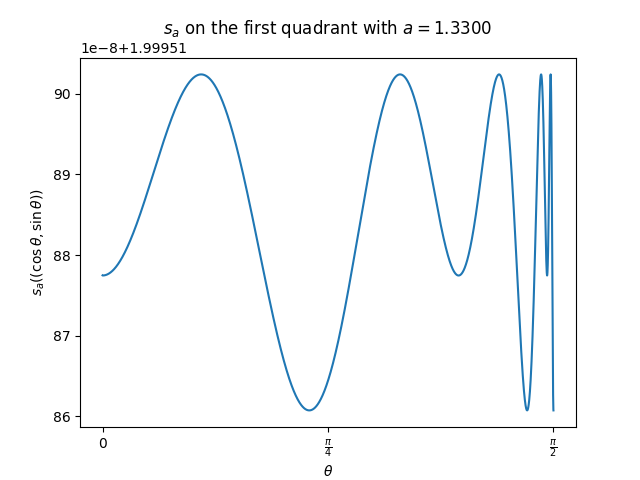}
\includegraphics[width=3.2in]{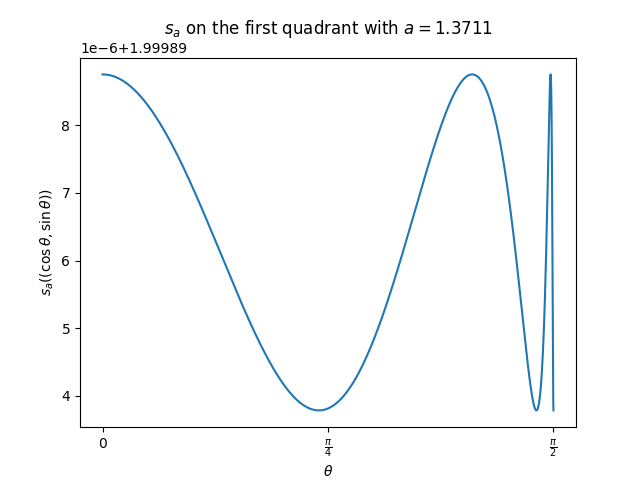}
\caption{
Each plot shows the value of $s_a$ on the portion of the circle in the first quadrant for a different fixed $a$.
In each case, 10000 points are sampled.
The 1-2-3-4-5 transition is observed in a neighborhood of $a = a_1 \approx 1.3299$.
From top left to bottom: Situation 1, $1 < a <a_1^-$; Situation 2, $a_1^- < a < a_1$; Situation 3, $a = a_1$; Situation 4, $a_1 < a < a_1^+$; Situation 5, $a_1^+ \le a \le a_2^-$.
}

\label{fig:star-side-lengths}
\end{figure}

Note that these five situations align with the different cases in Conjecture ~\ref{conj:main}.
Furthermore, we observed that as $a$ ranges through $(1, \sqrt{2})$, the function $s_a$ transitions between situations in the following order: 1-2-3-4-5-4-3-2-1.
Situations 1 and 5 are observed for the large majority of values of $a$.
The 5-4-3-2-1 portion of this transition can be found in a neighborhood of $a = a_2 \approx 1.4123$ at this link: 
(\footnote{https://github.com/nturillo/ellipse\textunderscore homotopy/blob/main/gifs/r\textunderscore vs\textunderscore theta\textunderscore a\textasciitilde1.4123.gif}).
In particular, by varying the eccentricity $a$ through $(1, \sqrt{2})$, the graphs like those in Figure~\ref{fig:star-side-lengths} suggest the following:
\begin{itemize}
    \item If $1 < a \le a_1^-$ or $a_1^+ \le a \le a_2^-$ or $a_2^+ \le a < \sqrt{2}$ then $s_a$ is in Situation 1 or 5.
    \item If $a = a_1$ or $a = a_2$ then $s_a$ is in Situation $3$
    \item If $a_1^- < a < a_1$ or $a_1 < a < a_1^+$ or $a_2^- < a < a_2$ or $a_2 < a < a_2^+$ then $s_a$ is in Situation 2 or 4.
\end{itemize} 
Note that each of these bullet points corresponds to a case in Theorem~\ref{thm:main}.

The case of $5$-pointed stars (winding fraction $\frac{2}{5}$) is more complicated than the case in~\cite{AAR} with triangles (winding fraction $\tfrac{1}{3}$), which is why we are unable to get exact formulas for the larger scales $r$ when the changes in the homotopy type of the Vietoris--Rips complex $\vr{E_a}{r}$ occur.
The first difference is that $3$-pointed stars in~\cite{AAR} are simply equilateral triangles so all angles and sides are the same, whereas here no conclusion can be drawn regarding angles of the 5-pointed star; see Figure~\ref{fig:minimal-star}.
The second is that the global behavior (such as the number of extrema) of $s_a$ is dependent on the eccentricity $a$ in the $5$-pointed star case, but not in the triangle case.

The repository of code we created to generate these approximations and graphs can be found on Github~\cite{Turillo2025ellipse_homotopy}.
Most of the code was written in Python, although a few computationally intensive numerical approximations were done with C++.
%We used at least 1000 points in every calculation regarding the behavior of the step function to ensure no behavior was missed.

%\note{Maybe move Conjecture~\ref{conj:main} here}

\section{Proof of Theorem~\ref{thm:extrema}}
\label{sec:extrema}

In this section we will prove Theorem~\ref{thm:extrema}, a partial result towards Conjecture~\ref{conj:main}.
Theorem~\ref{thm:extrema} states that there are eccentricities $a_1\approx 1.3299$ and $a_2\approx 1.4123$ such that $s_{a_1}$ and $s_{a_2}$ each have at least twenty local minima and twenty local maxima.
A high level overview of the proof is as follows.
Using elimination theory, we will calculate values of $s_a$ for specific points, namely the north and east poles of the circle, $(1, 0), (0,1) \in S^1$.
Then, we apply the intermediate value theorem to prove there exist two eccentricities, $a_1, a_2 \in (1, \sqrt{2})$ such that
\begin{align*}
    s_{a_1}((1, 0)) = s_{a_1}((0,1)) \\
    s_{a_2}((1, 0)) = s_{a_2}((0,1)) 
\end{align*}
In words, these are eccentricities for which the star on the ellipse $E_a$ based at $(a,0)$ has equal diameter to the star based at $(0,1)$.
Finally, we will show $(1,0),(0,1)\in S^1$ must be extrema of $s_a$, and since they have identical value, there must be more extrema lying between them.

With the plan of using the intermediate value theorem, we first state some lemmas about the continuity of $s_a$.
This proof will rely on the continuity of function $F$, which is proven in Section~\ref{sec:JointContF}.

\begin{lemma}
\label{lemma:continuity of s_a}
The functions $s_a: S^1 \rightarrow \R$ are continuous for all $a \in (1, \sqrt{2})$.
\end{lemma}

\begin{proof}
Recall that $s_a$ is equivalently defined as
    \[
s_a(p)=\inf\left\{r\in \R: \sum_{i=1}^5 \vec{d}(F^i_a(p,r),F^{i+1}_a(p,r)) \ge 2\right\}.
    \]
For convenience, let $f_a: S^1 \times (0,2) \rightarrow \R$ denote
\[
f_a(p, r) = \sum_{i=1}^5 \vec{d}(F^i_a(p,r),F^{i+1}_a(p,r)).
\]
Let $p \in S_1$, and let $\epsilon > 0$ be given.
We claim that there is an open neighborhood of $p$ whose image under $s_a$ lies in $(s_a(p) - \epsilon, s_a(p) + \epsilon)$.

Observe that $F^5_a(p, s_a(p)) = p$.
Hence, by monotonicity, $u\coloneqq F^5_a(p, s_a(p) - \epsilon)$ lies slightly counter-clockwise of $p$, and $v\coloneqq F^5_a(p, s_a(p) - \epsilon)$ lies slightly clockwise of $p$.
Then, by continuity of $F^5_a$ (see Section~\ref{sec:JointContF}), there is an open neighborhood $U$ of $p$ with $U \subseteq (u,v)_{S^1}$
such that $F^5_a(x, s_a(p) - \epsilon)\in(u,p)_{S^1}$ and $F^5_a(x, s_a(p) + \epsilon)\in (p,v)_{S^1}$ for all $x \in U$.
% Draw a figure for the next version?

Let $s \in U$ lie clockwise of $p$.
Recall that $v$ is clockwise of $s$.
So by monotonicity of $F^5_a$, we have that $F^5_a(s, s_a(p) + \epsilon)$ lies clockwise of $s$.
Hence, $f_a(s, s_a(p) + \epsilon) > 2$ and so, $s_a(s) < s_a(p) + \epsilon$.
On the other hand, $F^5_a(s, s_a(p) - \epsilon)$ lies counter-clockwise of $p$ and hence counter-clockwise of $s$.
Thus, $f_a(s, s_a(p) - \epsilon) < 2$ and $s_a(s) > s_a(p) - \epsilon$.
A symmetric argument shows that if $s \in U$ lies counter-clockwise of $p$, then $s_a(s) \in (s_a(p) - \epsilon, s_a(p) + \epsilon)$.
Thus, $U$ is our desired open neighborhood and hence $s_a$ is continuous.
\end{proof}

\begin{lemma}
\label{lemma: continuity of s_N, s_E}
The functions $s_N\colon (1, \sqrt{2}) \rightarrow \R$ given by $s_N(a) = s_a((0, 1))$ and $s_E\colon (1, \sqrt{2}) \rightarrow \R$ given by $s_E(a) = s_a((1,0))$ are continuous.
\end{lemma}

\begin{proof}
Our proof here is much like that of Lemma~\ref{lemma:continuity of s_a}, albeit a little simpler.
Again consider the infimum definition of $s_a$, and let $f_a : S^1 \times (0,1)$ denote the sum in the infimum definition.
Let $a \in (1, \sqrt{2})$ and let $\epsilon > 0$ be given.
Fix $p = (0,1) \in S^1$.
We claim there is an open neighborhood of $a$ whose image under $s_N$ lies in $(s_N(a) - \epsilon, s_N(a) + \epsilon)$.

Again refer to Section~\ref{sec:JointContF} for details on $F$.
Observe that $F^5(p, s_N(a), a) = p$.
By monotonicity of $F$, we have that $F^5(p, s_N(a) - \epsilon, a)$ lies slightly clockwise of $p$.
Hence, by continuity of $F$, there is an open neighborhood $U_1$ of $a$ such that $F^5(p, s_N(a) - \epsilon, b)$ lies clockwise of $p$ for all $b \in U_1$.
Likewise, there is an open neighborhood $U_2$ of $a$ such that $F^5(p, s_N(a) + \epsilon, b)$ lies counter-clockwise of $p$ for all $b \in U_2$.
Set $U = U_1 \cap U_2$.
Then, if $b \in U$, we have
\begin{align*}
F^5(p, s_N(a) - \epsilon, b) &\text{ is clockwise of } p\\
F^5(p, s_N(a) + \epsilon, b) &\text{ is counter-clockwise of } p
\end{align*}
so that $f_b(p, s_N(a) - \epsilon) < 2$ and $f_b(p, s_N(a) + \epsilon) > 2$.
Hence, we have $s_N(b) \in (s_N(a) - \epsilon, s_N(a) + \epsilon)$, as desired.
The proof for $s_E$ is identical.
\end{proof}

Here, we use $N$ and $E$ to refer to the "north" and "east" poles of the ellipse, i.e.\ the points $(0,1)$ and $(a, 0)$, respectively.
We note that by symmetry, the "south" and "west" poles would have worked just as well.
We also need the following technical lemma.

\begin{lemma}
\label{lemma: stars > triangles}
Fix $a \in (1, \sqrt{2})$.
The diameter of any (5-pointed equilateral) star inscribed in $E_a$ is greater than the diameter of any triangle inscribed in $E_a$.
\end{lemma}

\begin{proof}
Suppose not.
Then, there exists an ellipse $E_a$ and a diameter $r > 0$ such that there exists both an inscribed star and inscribed triangle in $E_a$ of diameter $r$.
Consider the (infinite) cyclic graph $\vr{E_a}{r}$.
Let $s \in E_a$ be a point on a star of diameter $r$ and $t \in E_a$ be a point on a triangle of diameter $r$.
Then, $2 \lfloor \frac{m}{5} \rfloor \leq \gamma_{m}(s) \leq 2 \lceil \frac{m}{5} \rceil$ for every $m \in \N$, and so we have
    \[
        \lim_{m \rightarrow \infty} \frac{\gamma_m(s)}{m} = \frac{2}{5}
    \]
Likewise, $\lfloor \frac{m}{3} \rfloor \leq \gamma_m(t) \leq \lceil \frac{m}{3} \rceil$ for every $m \in \N$, and hence
    \[
        \lim_{m \rightarrow \infty} \frac{\gamma_m(t)}{m} = \frac{1}{3}
    \]
But, by~\cite[Lemma~5.2]{AAR}, we have
    \[
        \lim_{m \rightarrow \infty} \frac{\gamma_m(t)}{m} = \wf (\vr{E_a}{r}) = \lim_{m \rightarrow \infty} \frac{\gamma_m(s)}{m},
    \]
a contradiction since the left-hand and right-hand terms above are not equal.
\end{proof}

Now, we will invoke some computational tools from SageMath to find the desired eccentricities $a_1$ and $a_2$.

\begin{lemma}
\label{lemma:(3)}
There exist $a_1 \in (1.32, 1.34)$ and $a_2 \in (1.4, 1.414)$ such that $s_N(a_1) = s_E(a_1)$ and $s_N(a_2) = s_E(a_2)$.
\end{lemma}

Before the proof, we remark that empirical evidence suggests $a_1 \approx 1.3299$ and $a_2 \approx 1.4123$.

\begin{proof}
\label{proof:(3)}
    We first define a function $D \colon (1, \sqrt{2}) \rightarrow \R$ by $D(a) = s_N(a) - s_E(a)$.
    We will show that $D(1.32) > 0$ and $D(1.34) < 0$, and thus there is $a_1$ such that $D(a_1) = 0$, proving the result for $a_1$.
    Similarly, we will show $D(1.4) < 0$ and $D(1.414) > 0$.
    Our main tool will be two polynomials $P_N, P_E$ in $a$ and $r$ whose zeroes represent the diameters of the north pole stars and east pole stars, respectively.
    To construct these polynomials, we begin with two systems of equations in twelve variables: $a$, the eccentricity as usual; $r$, the diameter of a star; and ten variables corresponding to the 2 dimensional coordinates of each of the five points of the star.
    First, we describe a system of equations whose solution set encodes five pointed stars in the ellipse.
    \begin{equation} \label{eq:star}
    \begin{aligned}
        (x_0/a)^2 + y_0^2 = 0&, \quad (x_1 - x_0)^2 + (y_1 - y_0)^2 - r^2 = 0,\\
        (x_1/a)^2 + y_1^2 = 0&, \quad (x_2 - x_1)^2 + (y_2 - y_1)^2 - r^2 = 0,\\
        (x_2/a)^2 + y_2^2 = 0&, \quad (x_3 - x_2)^2 + (y_3 - y_2)^2 - r^2 = 0,\\
        (x_3/a)^2 + y_3^2 = 0&, \quad (x_4 - x_3)^2 + (y_4 - y_3)^2 - r^2 = 0,\\
        (x_4/a)^2 + y_4^2 = 0&, \quad (x_0 - x_4)^2 + (y_0 - y_4)^2 - r^2 = 0,\\
    \end{aligned}
    \end{equation}
    The equations on the left side encode the requirement that each point of the star lies on the ellipse.
    The equations on the right side encode the requirement that the star is equilateral with diameter $r$.
    The solution space to~\eqref{eq:star} is too large to reduce to a polynomial using SageMath.
    Hence, we reduce the space to only north pole or only east pole stars using the following two sets of equations.
    \begin{equation}\label{eq:northsym}
    \begin{aligned}
        y_1 - y_4 = 0&, \quad x_1 + x_4 = 0\\
        y_2 - y_3 = 0&, \quad x_2 + x_3 = 0\\
        x_0 = 0&, \quad y_0 - 1 = 0
    \end{aligned}
\end{equation}
\begin{equation}
\label{eq:eastsym}
    \begin{aligned}
        y_1 + y_4 = 0&, \quad x_1 - x_4 = 0\\
        y_2 + y_3 = 0&, \quad x_2 - x_3 = 0\\
        x_0 - a = 0&, \quad y_0 = 0
    \end{aligned}
\end{equation}
    These sets of equations restrict the first point of the star to be at the corresponding pole, and also ensure the additional symmetries present in either case.
    Using the elimination theory tools in SageMath, we can reduce our systems to two singular polynomial equations in $a$ and $r$.
    Here, $P_N$ is constructed from~\eqref{eq:star} and~\eqref{eq:northsym}, while $P_E$ is likewise constructed from~\eqref{eq:star} and~\eqref{eq:eastsym}.
    \begin{align*}
    P_N(a,r) &=  50625 a^{24} r^{14} - 432000 a^{24} r^{12} + 40500 a^{22} r^{14} - 576000 a^{24} r^{10} + 1728000 a^{22} r^{12} \\
    & - 273150 a^{20} r^{14} + 10076160 a^{24} r^{8} - 14929920 a^{22} r^{10} + 2593920 a^{20} r^{12} - 63900 a^{18} r^{14} \\
    & - 983040 a^{24} r^{6} - 1474560 a^{22} r^{8} + 14059520 a^{20} r^{10} - 4603904 a^{18} r^{12} + 488815 a^{16} r^{14} \\
    & - 62914560 a^{24} r^{4} + 161218560 a^{22} r^{6} - 128696320 a^{20} r^{8} + 35885056 a^{18} r^{10} - 3518208 a^{16} r^{12} \\
    & - 91800 a^{14} r^{14} - 125829120 a^{22} r^{4} + 172359680 a^{20} r^{6} - 42663936 a^{18} r^{8} - 8330240 a^{16} r^{10} \\
    & + 3914752 a^{14} r^{12} - 201956 a^{12} r^{14} - 62914560 a^{20} r^{4} + 14417920 a^{18} r^{6} + 23478272 a^{16} r^{8} \\
    & - 9973760 a^{14} r^{10} + 703744 a^{12} r^{12} + 5480 a^{10} r^{14} - 11468800 a^{16} r^{6} + 5865472 a^{14} r^{8} \\
    & + 247808 a^{12} r^{10} - 274432 a^{10} r^{12} + 34991 a^{8} r^{14} - 802816 a^{12} r^{8} + 356352 a^{10} r^{10} \\
    & - 100736 a^{8} r^{12} + 9316 a^{6} r^{14} + 38400 a^{8} r^{10} - 10752 a^{6} r^{12} + 1026 a^{4} r^{14} - 384 a^{4} r^{12} \\
    & + 52 a^{2} r^{14} + r^{14}
    \end{align*}
    \begin{align*}
    P_E(a,r) &= a^{24} r^{14} + 52 a^{22} r^{14} - 384 a^{22} r^{12} + 1026 a^{20} r^{14} - 10752 a^{20} r^{12} \\
    & + 9316 a^{18} r^{14} + 38400 a^{20} r^{10} - 100736 a^{18} r^{12} + 34991 a^{16} r^{14} + 356352 a^{18} r^{10} \\
    & - 274432 a^{16} r^{12} + 5480 a^{14} r^{14} - 802816 a^{18} r^{8} + 247808 a^{16} r^{10} + 703744 a^{14} r^{12} \\
    & - 201956 a^{12} r^{14} + 5865472 a^{16} r^{8} - 9973760 a^{14} r^{10} + 3914752 a^{12} r^{12} - 91800 a^{10} r^{14} \\
    & - 11468800 a^{16} r^{6} + 23478272 a^{14} r^{8} - 8330240 a^{12} r^{10} - 3518208 a^{10} r^{12} + 488815 a^{8} r^{14} \\
    & + 14417920 a^{14} r^{6} - 42663936 a^{12} r^{8} + 35885056 a^{10} r^{10} - 4603904 a^{8} r^{12} - 63900 a^{6} r^{14} \\
    & - 62914560 a^{14} r^{4} + 172359680 a^{12} r^{6} - 128696320 a^{10} r^{8} + 14059520 a^{8} r^{10} + 2593920 a^{6} r^{12} \\
    & - 273150 a^{4} r^{14} - 125829120 a^{12} r^{4} + 161218560 a^{10} r^{6} - 1474560 a^{8} r^{8} - 14929920 a^{6} r^{10} \\
    & + 1728000 a^{4} r^{12} + 40500 a^{2} r^{14} - 62914560 a^{10} r^{4} - 983040 a^{8} r^{6} + 10076160 a^{6} r^{8} \\
    & - 576000 a^{4} r^{10} - 432000 a^{2} r^{12} + 50625 r^{14}
    \end{align*}
    For a given $a_0 \in (1, \sqrt{2})$, if the corresponding north pole star has diameter $r_0$, then $a_0, r_0$ are a solution to $P_N$.
    A similar statement is true for an east pole star and $P_E$.
    Hence, to find the diameters of north and east pole stars, we can examine the roots of $P_N$ and $P_E$ for specific values of $a$.
    This is how we will find $D(1.32)$ and $D(1.34)$.
    Still using SageMath, we find the real roots of $P_N$ and $P_E$ numerically when $a = 1.32$ and when $a = 1.34$.
    One potential trouble is that a triangle is actually a valid solution to our system of equations.
    We know the diameter of a maximal triangle in $E_a$ is
    \begin{equation*}
    r_2(a)=\frac{4\sqrt{3}a^2}{3a^2+1}
    \end{equation*}
    Since by Lemma~\ref{lemma: stars > triangles} the diameter of a minimal 5-pointed star must be strictly greater than that of a maximal triangle, we can filter out roots not greater than $r_2(a)$.
    Doing so, we find that the only real roots of $P_N$ and $P_E$ greater than $r_2(a)$ when $a = 1.32$ are approximately
    \[1.99934602760558, \text{ and } 1.99934434212106,\]
    respectively.
    Hence, 
    \[D(1.32) \approx 1.68548452283979 \cdot 10^{-6}\]
    One might worry that the sign of $D(1.32)$ could be incorrect due to floating-point precision error.
    However, we actually use the \texttt{max\_diameter} argument of SageMath's \texttt{real\_roots} function to find a rational interval in which $D(1.32)$ is guaranteed to lie.
    Both endpoints of this interval are positive, so $D(1.32) > 0$.
    Using the same process, we find
    \begin{align*}
        D(1.34) \approx -1.62962280314538 \cdot 10^{-6} \\
        D(1.4) \approx -3.13760593217971 \cdot 10^{-6} \\
        D(1.414) \approx 5.834802545567896 \cdot 10^{-7} 
    \end{align*}

By Lemma~\ref{lemma: continuity of s_N, s_E}, the function $D$ is continuous.
So, by the intermediate value theorem, there are $a_1 \in (1.32, 1.34)$, $a_2 \in (1.4, 1.414)$ such that $D(a_1) = D(a_2) = 0$.
Then, $s_N(a_1) = s_E(a_1)$ and $s_N(a_2) = s_E(a_2)$, as desired.

\end{proof}

\begin{lemma}
\label{lemma: poles are extrema}
The points $(\pm 1, 0)$ and $(0, \pm 1)$ are extrema of $s_a$ for all $a \in (0, \sqrt{2})$.
\end{lemma}

\begin{proof}
We prove that $(1,0) \in S^1$ is an extremum of $s_a$; the arguments for the points $(-1, 0)$, $(0, 1)$, and $(0, -1)$ being extrema are similar.
Fix $a \in (0, \sqrt{2})$ and $r > 0$.
Then, any five pointed star with diameter $r$ on the ellipse $E_a$ gives a solution to the set of polynomial equations~\eqref{eq:star} from the proof of Theorem~\ref{thm:extrema}.
This is a set of ten polynomial equations in ten variables, $x_0, \dots, x_4, y_0, \dots, y_4$.
Hence, by Bezout's Theorem, there are finitely many solutions to the system in $\C$.
%\note{2025-11-6: Henry says: See Figure~\ref{fig:chatgpt} for ChatGPT-5-Pro's ``check'' of general position---is this potentially along the right lines?}
In particular, there are finitely many solutions in $\R$.
In words, for a fixed ellipse, there are only finitely many stars of any given diameter.
Hence, there is an $\epsilon > 0$ neighborhood
\[
    U = \{(\cos(\theta), \sin(\theta)) \colon | \theta | < \epsilon \}
\]
such that $s_a^{-1} (s_a(1, 0)) = \{(1, 0)\}$.
So, by continuity of $s_a$, we either have $s_a(\cos(\theta), \sin(\theta)) > s_a(1,0)$ for all $\theta \in (0, \epsilon)$, or $s_a(\cos(\theta), \sin(\theta)) < s_a(1,0)$ for all $\theta \in (0, \epsilon)$.
Without loss of generality, assume the former.
By symmetry of the geometric interpretation of $s_a$, we see that
\[
    s_a(\cos(\theta), \sin(\theta)) = s_a(\cos(-\theta), \sin(-\theta))
\]
for all $\theta \in (-\epsilon, \epsilon)$.
Then, $s_a(p) > s_a(1,0)$ for all $p \in U$ with $p \not = (1,0)$.
So, the point $(1,0)$ is the minimum of $s_a$ on $U$, and thus a local extrema of $s_a$ on all of $E_a$.
\end{proof}

%\begin{figure}[h]
%\begin{center}
%\includegraphics[width=2.7in]{screenshot1.png}
%\includegraphics[width=2.7in]{screenshot2.png}
%\end{center}
%\caption{}
%\label{fig:chatgpt}
%\end{figure}

We are now in a position to prove Theorem~\ref{thm:extrema}.

\begin{proof}
Let $a_1, a_2$ be the values provided by Lemma~\ref{lemma:(3)}.
We prove the result for $a = a_1$, but the proof is identical for $a = a_2$.
By Lemma~\ref{lemma: poles are extrema}, the points $(\pm 1, 0)$, $(0, \pm 1)$ are all extrema of $s_a$.
Furthermore, for each of these points there are four more points that are extrema of $s_a$: the four other points in the inscribed star.
It is clear geometrically that no star can contain multiple poles.
Since the inscribed star is unique at each point, these twenty points are all distinct.
Call these points
\[ v_0 \prec v_1 \prec \dots \prec v_{18} \prec v_{19}, \]
where all indices are taken modulo $20$. 
We have found twenty extrema of $s_a$, and our goal was to find twenty minima and twenty maxima.
The following claim completes our proof.
If $v_i$ is a local minimum of $s_a$, then there exists $v_i \prec w \prec v_{i+1}$ such that $w$ is a local maximum of $s_a$, where all indices are taken modulo $20$.
Similarly, if $v_i$ is a local minimum, then there exists $v_i \prec w \prec v_{i+1}$ such that $w$ is a local maximum.
%Also recall that for a fixed $r > 0$, $|s_a^{-1}(r)| < \infty$ (see proof of $\ref{lemma: poles are extrema}$), so that $s_a$ is not constant between any $v_i$ and $v_{i+1}$.

To prove our claim, let $0 \le i \le 19$.
Recall that by Lemma~\ref{lemma:(3)}, $s_a(v_l) = s_a(v_k)$ for all $0 \le l < k \le 19$.

Suppose $v_i$ is a minimum and $v_{i+1}$ is a maximum.
Then, there are $v_i \prec v_i' \prec v_{i+1}' \prec v_{i+1}$ such that $s_a(v_{i+1}') < s_a(v_{i+1}) = s_a(v_i) < s_a(v_i')$.
Hence, by continuity (\ref{lemma:continuity of s_a}), $s_a$ achieves both a minimum and maximum on the open arc of the circle between $v_i$ and $v_{i+1}$.
In particular, $s_a$ achieves a maximum, as desired.

The other cases are similarly and we describe them briefly.
If $v_i$ is a maximum and $v_{i+1}$ is a minimum, then we have $v_i \prec v_i' \prec v_{i+1}' \prec v_{i+1}$ with $s_a(v_i') < s_a(v_i) = s_a(v_{i+1}) < s_a(v_{i+1}')$.
So, again $s_a$ achieves a minimum on the open arc between $v_i$ and $v_{i+1}$.
If $v_i$ and $v_{i+1}$ are both minima, then we have $v_i \prec v_i' \prec v_{i+1}$ with $s_a(v_i') > s_a(v_i) = s_a(v_{i+1})$ and so $s_a$ achieves a maximum on the open arc between $v_i$ and $v_{i+1}$.
Similarly, if $v_i$ and $v_{i+1}$ are both maxima, then we have $v_i \prec v_i' \prec v_{i+1}$ with $s_a(v_i') < s_a(v_i) = s_a(v_{i+1})$, and so $s_a$ achieves a minimum on the open arc between $v_i$ and $v_{i+1}$.
In each of the four cases, if $v_i$ is a max, there is an additional min $w$, and if $v_i$ is a min, there is an additional max $w$.
Hence, there are at least twenty minima and twenty maxima, as desired.
\end{proof}
Finally, we remark that empirical evidence strongly suggests that each $v_i$ described above is actually a global minimum of $s_a$, and between each global minimum there is exactly one global maximum (see Situation 3 in Figure~\ref{fig:star-side-lengths}).

\section{Proof of Theorem~\ref{thm:main}}
\label{sec:maintheorem}
\subsection{Proof of Theorem~\ref{thm:main}}

We now prove our main Theorem~\ref{thm:main} regarding the homotopy types of $\vr{E_a}{r}$, assuming that Conjecture~\ref{conj:main} is true.
Thoughout this section, we let $1 < a < \sqrt{2}$.
Our proof is similar to the proof of~\cite[Theorem~7.3]{AAR}, and will use~\cite[Theorems~5.1 and~5.3]{AAR}.
We outline a sketch:
\begin{itemize}
\item the proof of the $S^3$ homotopy types will use~\cite[Theorem~5.1]{AAR}
with $l=1$; 
\item the proof of the $S^4$ homotopy type will use~\cite[Theorem~5.3]{AAR} with $l=2$, $P=0$, and $F=2$; 
\item the proof of the $\bigvee^3 S^4$ homotopy type will use~\cite[Theorem~5.3]{AAR} with $l=2$, $P=0$, and $F=4$; 
\item the proof of the $S^5$ homotopy types will use~\cite[Theorem~5.1]{AAR}
with $l=2$.
\end{itemize}
We will consider the $1$-skeleton of the simplicial complex $\vr{E_a}{r}$ as a directed cyclic graph as well.
Specifically, the directed cyclic graph induced by the Vietoris--Rips complex of an ellipse is closed and continuous; see Section~\ref{sec:preliminaries}.

%(\footnote{
%\note{
%There will also be extensions for when the $S^5$, and $S^6$ homotopy types appear depending on when $7$-pointed stars wrapping 3 times around appear.
%More generally, there will also be extensions for when the $S^{2l+1}$, and $S^{2l+2}$ homotopy types appear depending on when $(2l+3)$-pointed stars wrapping $l+1$ times around appear.
%}
%})

The proof the main theorem depends upon further definitions regarding when our Vietoris--Rips complex is singular.
We defer this discussion to Section~\ref{sec:singularcase} along with the proofs of the corresponding parts of the following lemmas.

\begin{lemma}
\label{lemma:MainTheorempt1}
If $1 < a \le a_1^-$ or $a_1^+ \le a \le a_2^-$ or $a_2^+ \le a < \sqrt{2}$, then 
\[\vr{E_a}{r}\simeq\begin{cases}
S^3 & \text{if }r_2(a)<r\le r_3(a) \\
S^4 & \text{if }r_3(a)<r\le r_4(a) \\
S^5 & \text{if }r_4(a)<r\le r_5(a).
\end{cases}\]
\end{lemma}

\begin{proof}
Let $r_2(a)<r<r_3(a)$.
Then $\frac{1}{3}<\wf(\vr{E_a}{r})<\frac{2}{5}$.
In the case that $r= r_3(a)$, because of our choice of $a$, there will be exactly $10$ points (or $2$ stars) that attain a winding fraction of $\frac{2}{5}$.
So by~\cite[Theorem 5.1]{AAR}, we have $\vr{E_a}{r} \simeq S^3$.

See Lemma~\ref{lemma:singularcase1} for the $S^4$ case.

The case when $r_4(a)<r\leq r_5(a)$ follows similarly argument as when $r_3(a)<r\leq r_5(a)$.
Specifically, we can see that $\frac{2}{5}< \wf(\vr{E_a}{r})\leq \frac{3}{7}$ and so by~\cite[Theorem 5.1]{AAR} we have $\vr{E_a}{r} \simeq S^5$.
\end{proof}

\begin{lemma}
\label{lemma:MainTheorempt2}
If $a = a_1$ or $a = a_2$, then 
\[\vr{E_a}{r}\simeq\begin{cases}
S^3 & \text{if }r_2(a)<r\le r_3(a) \\
\bigvee^3 S^4 & \text{if }r_3(a)<r\le r_4(a) \\
S^5 & \text{if }r_4(a)<r\le r_5(a).
\end{cases}\]
\end{lemma}

\begin{proof}
The cases $r_2(a)<r\leq r_3(a)$ and $r_4(a) < r\leq r_5(a)$ follow an argument similar to Lemma~\ref{lemma:MainTheorempt1}.
For the case when $r_3(a)<r\leq r_4(a)$, see Lemma~\ref{lemma:singularcase2}.
\end{proof}

\begin{lemma}
\label{lemma:MainTheorempt3}
If $a_1^- < a <a_1$ or $a_1 < a < a_1^+$ or $a_2^- < a < a_2$ or $a_2 < a < a_2^+$, then there is an additional real function $r_{7/2}$ (local minimum side-length of $5$-pointed stars) so that
\[\vr{E_a}{r}\simeq\begin{cases}
S^3 & \text{if }r_2(a)<r\le r_3(a) \\
S^4 & \text{if }r_3(a)<r\le r_{7/2}(a) \\
\bigvee^3 S^4 & \text{if }r_{7/2}(a)<r\le r_4(a) \\
S^5 & \text{if }r_4(a)<r\le r_5(a).
\end{cases}\]
\end{lemma}

\begin{proof}
The cases where $r_2(a)<r\leq r_3(a)$ and $r_4(a)<r\leq r_5(a)$ follow similarly as in Lemma~\ref{lemma:MainTheorempt1}.
For the remaining cases, see Lemma~\ref{lemma:singularcase3}.
\end{proof}

We are now ready to prove the main result of this section.

\begin{proof}[Proof of Theorem~\ref{thm:main}]
The homotopy types of the Vietoris--Rips complex $\vr{E_a}{r}$ in the relevant ranges of scale $r$ have been shown in Lemmas~\ref{lemma:MainTheorempt1},~\ref{lemma:MainTheorempt2}, and~\ref{lemma:MainTheorempt3}.

We now prove the homotopy equivalence of the inclusion $\vr{E_a}{r}\hookrightarrow \vr{E_a}{\tilde{r}}$.
Let $a\in (1,\sqrt{2})$ given.
Suppose $r,\tilde{r}$ are real numbers such that $r_2(a) < r\leq \tilde{r} \leq r_3(a)$ or $r_4(a)<r\leq \tilde{r}<r_5(a)$.
We observe that 
\[\tfrac{1}{3} <\wf( \vr{E_a}{r}) \leq \wf( \vr{E_a}{\tilde{r}})\leq \tfrac{2}{5}\] 
in the first case and
\[
\tfrac{2}{5}  <\wf( \vr{E_a}{r}) \leq \wf( \vr{E_a}{\tilde{r}})\leq\tfrac{3}{7} 
\]
in the second case.
In either of these cases, by~\cite[Theorem 5.1]{AAR}, the inclusion induces a homotopy equivalence.

Now suppose $1 < a \le a_1^-$, $a = a_1$, $a_1^+ \le a \le a_2^-$, or $a = a_2$.
Let $r,\tilde{r}$ be such that $r_3(a)<r\le \tilde{r}\le r_4(a)$ or $a_2^+ \le a < \sqrt{2}$.
Fix a finite subset $W\in \widetilde{\fin}(\vr{E_a}{r})$.
From the proof of~\cite[Theorem~5.3]{AAR}, we can deduce that 
\[\text{colim}_{W\in \widetilde{\fin}(\vr{E_a}{r})} \vr{W}{r} \simeq \text{hocolim}_{W\in \widetilde{\fin}(\vr{E_a}{r})}\vr{W}{r}\] 
and thus the inclusion $\iota_r\colon \vr{W}{r} \hookrightarrow \vr{E_a}{r}$ induces a homotopy equivalence.
By~\cite[Lemma 5.9]{AAR}, we can find some $\widetilde{W}\supseteq W$ such that $\widetilde{W}\in \widetilde{\fin}(\vr{E_a}{\tilde{r}})$.
Using a similar argument as before, the inclusion $\iota_{\tilde{r}}\colon\vr{\widetilde{W}}{\tilde{r}}\hookrightarrow\vr{E_a}{\tilde{r}}$ is a homotopy equivalence.
By~\cite[Lemma 5.8]{AAR}, we can conclude that the inclusion $\iota_W\colon \vr{W}{r} \hookrightarrow\vr{\widetilde{W}}{\tilde{r}}$ is a homotopy equivalence.
Hence, the following diagram commutes:

\[\begin{tikzcd}
	{\text{VR}(W;r)} && {\text{VR}(E_a;r)} \\
	\\
	{\text{VR}(\widetilde{W};\tilde{r})} && {\text{VR}(E_a;\tilde{r})}
	\arrow["{\iota_r}"', from=1-1, to=1-3]
	\arrow["{\iota_W}", from=1-1, to=3-1]
	\arrow["{\iota_{E_a}}"', from=1-3, to=3-3]
	\arrow["{\iota_{\tilde{r}}}", from=3-1, to=3-3]
\end{tikzcd}\]
which shows the inclusion $\iota_{E_a}\colon \vr{E_a}{r}\hookrightarrow\vr{E_a}{\tilde{r}}$ is a homotopy equivalence.

Suppose $a_1^- < a <a_1$, $a_1 < a < a_1^+$, $a_2^- < a < a_2$, or $a_2 < a < a_2^+$, and let $r,\tilde{r}$ be such that $r_3(a)<r\le \tilde{r}\le r_{7/2}(a)$ or $r_{7/2}(a)<r\le \tilde{r}\le r_4(a)$.
The proof that the inclusion $\vr{E_a}{r}\hookrightarrow\vr{E_a}{\tilde{r}}$ is a homotopy equivalence follows by a similar argument as the above case.

Suppose $a_1^- < a <a_1, a_1 < a < a_1^+, a_2^- < a < a_2$, or $a_2 < a < a_2^+$.
We now prove that for $r_3(a)<r\leq r_{7/2}(a) < \tilde{r} \leq r_4(a)$, the inclusion $\vr{E_a}{r}\hookrightarrow\vr{E_a}{\tilde{r}}$ induces a rank 1 map on $4$-dimensional homology.
Fix some $W\in \widetilde{\fin}(\vr{E_a}{r})$.
By cofinality, we can find some $\tilde{W}\in \widetilde{\fin}(\vr{E_a}{\tilde{r}})$ such that $W\subseteq \tilde{W}$.
As before, the following diagram commutes:
\[\begin{tikzcd}
	{\text{VR}(W;r)} && {\text{VR}(E_a;r)} \\
	\\
	{\text{VR}(\widetilde{W};\tilde{r})} && {\text{VR}(E_a;\tilde{r})}
	\arrow["{\iota_r}"', from=1-1, to=1-3]
	\arrow["{\iota_W}", from=1-1, to=3-1]
	\arrow["{\iota_{E_a}}"', from=1-3, to=3-3]
	\arrow["{\iota_{\tilde{r}}}", from=3-1, to=3-3]
\end{tikzcd}\]
As before, each of the $\iota$ maps are homotopy equivalences.
Applying homology, we get that the following diagram also commutes:
\[\begin{tikzcd}
	{H_4(\text{VR}(W;r),\mathbb{F})} && {H_4(\text{VR}(E_a;r),\mathbb{F})} \\
	\\
	{H_4(\text{VR}(\widetilde{W};\tilde{r}),\mathbb{F})} && {H_4(\text{VR}(E_a;\tilde{r}),\mathbb{F})}
	\arrow["{{\iota_r^*}}"', from=1-1, to=1-3]
	\arrow["{{\iota_W^*}}", from=1-1, to=3-1]
	\arrow["{{\iota_{E_a}^*}}"', from=1-3, to=3-3]
	\arrow["{{\iota_{\tilde{r}}^*}}", from=3-1, to=3-3]
\end{tikzcd}\]
Since each $\iota$ are homotopy equivalences, then each $\iota^*$ are isomorphisms between homology groups.

Now we will show that $H(i^*_{W})=2$.
Notice that at the scale $r$, there are only two periodic orbits in $\vr{W}{r}$.
They can be found by considering the permanently fast points $I_0,I_1$ in $\vr{E_a}{r}$ and consider that the intersections $\vr{W}{r}\cap I_0$ and $\vr{W}{r} \cap I_1$.
The points in each of these intersections are invariant under our dynamical system.
Further, since there are only finitely many points, each of these intersections contain exactly one periodic orbit.
However note that for $\vr{\tilde{W}}{\tilde{r}}$ has four periodic orbits as a consequence of Lemma~\ref{lemma:MainTheorempt3}.
Hence, the inclusion maps two periodic orbits to two periodic orbits.
By~\cite[Proposition 4.2]{AAR}, $\text{rank}(i_W^*) = 2-1 = 1$.
\end{proof}

\subsection{Details regarding the singular case}
\label{sec:singularcase}

The definition of the parameter $l$ may be defined in more generality, however since our principal investigation is focused on $5$ pointed stars, we let $l=2$.

Assume we are given a singular Vietors--Rips complex.
Since our Vietoris--Rips complex is generated by open balls, the set of fast points is the disjoint union of finitely many open intervals.
In fact, each of the open intervals appear as open neighborhoods about the vertices of the global minimal stars.
Hence, for each $5$ pointed star, there are $5$ total such open sets.
Letting $I$ be the union of these sets, we find that $I$ is invariant under our dynamical system.
Thus, we call $I$ an invariant set of fast points.
We define $F_{a,r}$ to count the number of disjoint invariant sets of fast points.
%\note{We note that the calculation for $F$ can be obtained by considering the $10$ intervals from $I_1$ and $I_2$ each as equivalence classes of permanently fast points and applying~\cite[Lemma 5.17]{AAR} to see that $F= \frac{10}{2l+1} = 2$.}

Finally, let $P_{a,r}$ count the number of periodic orbits.
However note that since our because our complex is generated by open balls, we never attain a periodic orbit in the complex.
Hence, $P_{a,r}=0$ for all $a,r$.

We will often omit the $a$ and $r$ when the context is clear.
The following lemmas will be used to prove the theorem.

\begin{lemma}[Singular case of Lemma~\ref{lemma:MainTheorempt1}]
\label{lemma:singularcase1}
Let $1 < a \le a_1^-$ or $a_1^+ \le a \le a_2^-$ or $a_2^+ \le a < \sqrt{2}$, and $r_3(a)<r\leq r_4(a)$ then $\vr{E_a}{r}\simeq S^4$
\end{lemma}

\begin{proof}
Let $r_3(a)<r\leq  r_4(a)$.
By Conjecture~\ref{conj:main} and the intermediate value theorem, there exist open sets 
\[
I_1 = (p_1^1(r), p_2^1(r))_{E_a} \cup (p_1^2(r),p_2^2(r))_{E_a}\cup \ldots \cup (p_1^5(r),p_2^5(r))_{E_a} \text{ and}
\]
\[
I_2 = (p_3^1(r), p_4^1(r))_{E_a} \cup (p_3^2(r),p_4^2(r))_{E_a}\cup \ldots \cup (p_3^5(r),p_4^5(r))_{E_a}
\]
that are each invariant sets of permanently fast points for $\vr{E_a}{r}$.
We point out here that the sets $I_1$ and $I_2$ correspond to the union of open neighborhoods about the $10$ points with an inscribed star that has radius equal to the global minima.
Hence, we have $l=2$,$P=0$ and $F=2$.
By~\cite[Theorem 5.3]{AAR} we can conclude $\vr{E_a}{r} = \bigvee^{0+2-1}S^4 = S^4$.
\end{proof}

\begin{lemma}[Singular case of Lemma~\ref{lemma:MainTheorempt2}]
\label{lemma:singularcase2}
If $a = a_1$ or $a = a_2$, and $r_3(a)<r\leq r_4(a)$ then $\vr{E_a}{r} \simeq \bigvee^3S^4$.
\end{lemma}

\begin{proof}
By Conjecture~\ref{conj:main}, there are $20$ global minima.
As with Lemma~\ref{lemma:singularcase1}, there are $4$ sets $I_1,I_2,I_3,$ and $I_4$ which are each the union of $5$ intervals, each containing a unique point whose inscribed star has radius equal to the global minima radius.
Thus, there are $4$ invariant sets of permanently fast points.
By applying~\cite[Theorem 5.3]{AAR} with $P=0$, $F=4$ and $l=2$ then $\vr{E_a}{r} \simeq \bigvee^3S^4$.
\end{proof}

\begin{lemma}[Singular case of Lemma~\ref{lemma:MainTheorempt3}]
\label{lemma:singularcase3}
If $a_1^- < a <a_1$ or $a_1 < a < a_1^+$ or $a_2^- < a < a_2$ or $a_2 < a < a_2^+$, then there is an additional real function $r_{7/2}$ (local minimum side-length of $5$-pointed stars) so that
\[\vr{E_a}{r} \simeq \begin{cases}
S^4 & \text{if }r_3(a)<r\le r_{7/2}(a) \\
\bigvee^3 S^4 & \text{if }r_{7/2}(a)<r\le r_4(a) 
\end{cases}\]
\end{lemma}

\begin{proof}
By Conjecture~\ref{conj:main}, there exist $10$ points, corresponding to the points on two inscribed stars, which are locally minimum with respect to the side length of the inscribed star.
Hence, there exists a function $r_{7/2}:(1,\sqrt{2})\to (0,2)$ which map an eccentricity value $a$ to the radius value which is a local minimum.

Suppose $r_3(a)<r\leq r_{7/2}(a)$.
Then there exists only $2$ invariant sets of permanently fast points; they are exactly those that contain the sets that contain the points of the two (globally) minimally inscribed stars.
By a similar argument as Lemma~\ref{lemma:MainTheorempt1}, we have $\vr{E_a}{r} \simeq S^4$.

Suppose otherwise that $r_{7/2}(a)<r\leq r_4(a)$.
Then there exists only $4$ invariant sets of permanently fast points.
Two of these sets contain the $5$ points which are on the two (globally) minimally inscribed stars.
Furthermore, since $r_{7/2}(a)<r$ then there are two sets that contain the $10$ points associated with the two stars of the (locally) minimally inscribed stars.
These are also invariant sets of permanently fast points.
Therefore, by a similar reasoning as in Lemma~\ref{lemma:MainTheorempt2}, we deduce $\vr{E_a}{r} \simeq \bigvee^3S^4$.
\end{proof}

\section{Joint continuity}
\label{sec:JointContF}

\subsection{Definition of $F$}
We use many functions to analyze the Vietoris--Rips complex of an ellipse.
In this section, we provide the rigorous definitions for them, along with proofs of continuity.

Let $d : \R^2 \times \R^2 \rightarrow \R$ be the Euclidean distance function and recall from Section~\ref{sec:preliminaries} that $\varphi_a \colon S^1\to E_a$ is the natural homeomorphism between a circle and an ellipse of eccentricity $a$ defined by $\varphi_a(\cos \theta,\sin \theta) = (a \cos\theta, \sin \theta)$.
As mention in Section~\ref{sec:preliminaries}, whenever a function has domain or codomain $S^1$ but refers to the ellipse $E_a$, assume this function is composed with $\varphi_a$ or its inverse, as necessary.
%In the following, we utilize some definitions from~\cite{AAR}.
%Given an ellipse $E_a$ of small eccentricity, let $(p,q)_{E_a} = \{z\in Y: p\prec z \prec q \prec p\}$.
%We let $[p,q]_{E_a}$ be defined similarly with $\preceq$ instead of $\prec$.
%For convenience, let $[p,p]_{E_a} = \{p\}$.
%\note{Nicco says: I think these are defined in the intro}
%We will often omit the notation for $E_a$ when the context is clear.
For any ellipse $a \in (1, \sqrt{2})$ define a function $h_a\colon S^1 \to S^1$ which maps a point $p$ on the ellipse $E_a$ to the other unique point intersected by a normal line of $E_a$ at point $p$.
By~\cite[Lemma~6.4]{AAR}, $h_a$ is a bijection for $a\in (1,\sqrt{2})$.
Then define $F\colon S^1\times(0,2)\times(1,\sqrt{2})\to S^1$ by 
\[F(p,r,a) = \argmax_{q\in [p,h^{-1}_a(p)]} \{d(\varphi_a(p),\varphi_a(q))~|~d(\varphi_a(p),\varphi_a(q)) \le r, p \preceq q\}.\]
Intuitively, $F$ is the function that takes a step clockwise to the point on the ellipse $E_a$ that is a Euclidean distance $r$ away from $p$.

Next, we define an extension of $F$ which allows us to take $n$-fold compositions.
Define $F^n$ recursively as follows.
First, set $F^1(p,r,a) = F(p,r,a)$.
Next, define $F^{n+1}(p,r,a) = F(F^n(p,r,a),r,a)$ for all $n >= 1$.

If we fix a point $p$ on an ellipse of eccentricity $a$ and a winding fraction $\frac{\alpha}{\beta}$ with $\alpha,\beta\in\N$, there is exactly one radius $r$ such that $\gamma_\beta(p)=\alpha$ when $p$ is acted on by our dynamical system $F$.
Hence, we can define the side length function $s: S^1\times (1,\sqrt{2})\times (0,1) \to \R$ by
\[s(p,a,\tfrac{\alpha}{\beta}) =\inf\left\{r\in \R~|~\sum_{i=1}^{\beta}\vec{d}(F^i(p,r,a),F^{i+1}(p,r,a)) \geq \alpha\right\}.\]
In this paper, we are principally interested in the case when the winding fraction is $\frac{2}{5}$.
We can thus redefine, with some abuse of notation, $s \colon S^1\times (1,\sqrt{2}) \to \R$ as our previously defined $s$ function, but fixing $\frac{\alpha}{\beta}=\frac{2}{5}$.
This definition is exactly the one we provide in Section~\ref{sec:sidelength}.

\subsection{Continuity of $F$}

We will prove that $F\colon S^1\times(0,2)\times(1,\sqrt{2})\to S^1$ is continuous (i.e., jointly continuous in all three variables).
First, let us consider a function studied in~\cite{AAR}.
The authors define a function $G\colon E_a\times (0,2) \to E_a$ where $E_a$ is an ellipse of small eccentricity.
Specifically, the function is defined as 
\[G(p,r) =\argmax_{q\in [p,h^{-1}(p)]_{E_a}} \{d(p,q) ~ | ~d(p,q)\leq r\}.\]
Our function $F$ is defined very similarly and is a natural extension of the function $G$.
With $G$, we are fixing beforehand a given ellipse $E_a$ with small eccentricity.
However, with $F$, we are allowing the eccentricity of the ellipse to vary along with the point and radius as necessary.
Thus with this view, we note that for a fixed ellipse $E_a$ and our function $G$ defined with respect to $E_a$, then $G(-,-) = F(-,-,a)$.
With such a view, it is clear that results proven about $G$ will automatically hold for our function $F$, with fixed eccentricities $a$.
In order to prove the result we require, we need to first show $F$ is continuous with respect to eccentricity.

\begin{lemma}
\label{ray:FcontIna}
For a fixed point $p\in S^1$ and $0<r<2$, the function $F(p,r,-)\colon (1,\sqrt{2})\to S^1$ is continuous.
\end{lemma}

\begin{proof}
Let $p\in S^1$, $r\in (0,2)$, and $a_1,a_2\in (1,\sqrt{2})$ be given.
Let us restrict the codomain of $F$ to the subspace $X=(F(p,r,\sqrt{2}),F(p,r,1))_{S^1}$ so that there is an order-preserving homeomorphism from $X$ to $\R$.
We note here that $X$ is a connected proper subspace of $S^1$.
    
First, we will show that the function is monotonic in $a$ with respect to the $\preceq$ ordering in the codomain.
Without loss of generality suppose  $a_1<a_2$.
Let $p_1 = F(p,r,a_1)$ and $p_2= F(p,r,a_2)$.
We will parameterize $p$ as $(\cos t,\sin t)$ and $p_i$ as $(\cos t_i,\sin t_i)$.

Then 
\[d(\varphi_{a_1}(p),
\varphi_{a_1}(p_1))= \sqrt{a_1^2(\cos t - \cos t_1)^2 + (\sin t- \sin t_1)^2} = r = d(\varphi_{a_2}(p), \varphi_{a_2}(p_2)).\]
However, consider 
\[d(\varphi_{a_2}(p),\varphi_{a_2}(p_1)) = \sqrt{a_2^2(\cos t - \cos t_1)^2 + (\sin t- \sin t_1)^2} > r = d(\varphi_{a_2}(p), \varphi_{a_2}(p_2)).\]
By definition of $F$, we have $p\preceq p_1$ and $p\preceq p_2$.
Since $d(\varphi_{a_2}(p),\varphi_{a_2}(p_1)) > d(\varphi_{a_2}(p),\varphi_{a_2}(p_2))$ then $p_2 \prec p_1$ because $\varphi_a$ is order-preserving.
Hence, $F$ is strictly monotonically decreasing with respect to $a$.

We now claim that $F(p,r,-)$ is surjective onto $X$.
Let $q$ be a given point in $X$.
We note that as $a$ increases, the distance between $\varphi_a(p)$ and $\varphi_a(q)$ increases.
Thus consider the function $g:(1,\sqrt{2}) \to \R$ by $g(a) = d(\varphi_a(p),\varphi_a(q)) - r$.
Notice that since $F(p,r,-)$ is monotonic, there is $\epsilon > 0$ such that $F(p,r,\sqrt{2} - \epsilon) \prec q \prec F(p,r,1 + \epsilon)$.
So notably, 
\[\varphi_a(F(p,r,\sqrt{2} - \epsilon)) \prec \varphi_a(q) \prec \varphi_a(F(p,r,1 + \epsilon))\]
for any $a\in (1,\sqrt{2})$.
Hence, $g(1 + \epsilon) < 0$ since $ \varphi_{1 + \epsilon}(q) \prec \varphi_{1 + \epsilon}(F(p,r,1 + \epsilon))$.
Likewise, $g(\sqrt{2} - \epsilon) >0$ since $\varphi_{\sqrt{2} - \epsilon} (F(p,r,\sqrt{2} - \epsilon)) \prec \varphi_{\sqrt{2} - \epsilon}(q)$.
By the intermediate value theorem, there exists some $a \in (1 + \epsilon, \sqrt{2} - \epsilon)$ such that $g(a) = 0$, meaning also that $F(p,r,a) = q$.
This shows $F(p,r,-)$ is indeed surjective onto $X$.

Now, since there is clearly an order-preserving homeomorphism from $X$ to an interval in $\R$, we conclude that $F(p,r,-)$ is continuous~\cite{ProofWiki_SurjectiveMonotoneFunctionIsContinuous}.
% We'll want to replace the above with a textbook reference, likely.
\end{proof}

\begin{theorem}
$F$ is jointly continuous in all variables.
\end{theorem}

\begin{proof}
The continuity of $F$ follows naturally from an extension of the proof of~\cite[Lemma 6.8]{AAR}.
Specifically, we note that for each fixed value eccentricity value $a$, both $p$ and $r$ are monotonic with respect to $a$.
From~\cite[Lemmas~6.6 and~6.7]{AAR} and Lemma~\ref{ray:FcontIna}, we have that $F$ is separately continuous in each variable.
Hence by~\cite[Proposition 2]{kruse1969joint}, $F$ is jointly continuous in all three variables.
\end{proof}

\begin{corollary}
\label{cor:F'-cont}
$F^i$ is jointly continuous in all variables for all $i$.
\end{corollary}

\section{Conclusion}
\label{sec:conclusion}

While we did not supply a proof for Conjecture~\ref{conj:main}, we have supplied numerical evidence to support our belief in the conjecture.
Furthermore, should Conjecture~\ref{conj:main} be true, then Theorem~\ref{thm:main} follows, which extends the known homotopy type of $\vr{E_a}{r}$ at scales $r$ greater than those studied in~\cite{AAR}.
%Specifically, if $1 < a \le a_1^-$, $a_1^+ \le a \le a_2^-$, or $a_2^+ \le a < \sqrt{2}$, then
%$$
%\vr{E_a}{r} \simeq \begin{cases}
%    S^3 ~\text{if}~r_2(a)<r\leq r_3(a) \\
%    S^4 ~\text{if}~r_3(a)<r\leq r_4(a) \\ 
%    S^5 ~\text{if}~r_4(a)<r\leq r_5(a).
%\end{cases}
%$$
%More surprisingly, when $a = a_1$ or $a = a_2$, we observe strange behavior as described in~\ref{thm:extrema}, which leads us to conclude that in this case,
%$$
%\vr{E_a}{r} \simeq \begin{cases}
%    S^3 ~\text{if}~r_2(a)< r \leq r_3(a) \\ 
%    \bigvee^3 S^4 ~\text{if}~r_3(a) < r \leq r_4(a) \\
%    S^5 ~\text{if}~r_4(a) < r \leq r_5(a)
%\end{cases}
%$$

We end with a list of open questions.

\begin{enumerate}

\item
The case of 3-pointed stars in~\cite{AAR} has simpler behavior than the case of 5-pointed stars considered here.
What is the behavior of inscribed 7-pointed, 9-pointed, 11-pointed, \ldots stars in an ellipse? 
Does a general pattern emerge relating the number of extrema to the number of points on the star?
Answers to this question would allow one to understand the homotopy types of $\vr{E_a}{r}$ for scales $r$ even larger than those considered in this paper.

\item
We observe critical behavior in the side-length function when the eccentricity of the ellipse $E_a$ is $a = a_1 \approx 1.3299$ or $a = a_2 \approx 1.4123$.
What is special about these two particular values of eccentricity $a$?
Is there a closed form expression for these two critical values?

\item
A proof of Conjecture~\ref{conj:main} is still essential to complete the proof of the homotopy types of $\vr{E_a}{r}$, at the scales described above.
What tools would be useful for such a proof?

\item
Let $E_{a,b}\coloneqq\{(x,y,z)\in\R^3~|~(x/a)^2+(y/b)^2+z^2=1\}$ be an ellipsoid with $a\ge b\ge 1$ and $a>1$.
What is the smallest diameter $r$ (as a function of $a$ and $b$) such that some tetrahedron inscribed in $E_a$ of diameter at most $r$ contains the origin $(0,0,0)$?
What are the homotopy types of the Vietoris--Rips complexes $\vr{E_{a,b}}{r}$?

\end{enumerate}

\bibliographystyle{plain}
\bibliography{VietorisRipsComplexesOfEllipsesAtLargerScales}

%\appendix

% For the other missing sections, see "Appendices-Old-2025-11-6.tex"

% For the old appendices, see "Appendices-Old-2025-10-9.tex"

% For the old comments, see "old-2025-11-9.tex"

\end{document}